\documentclass{article}
\usepackage{cite}
\usepackage{amsmath,amssymb,amsthm}
\usepackage{titlesec,hyperref}
\usepackage{color}
\usepackage{fancyhdr}
\usepackage[margin=3cm]{geometry}
\usepackage{graphicx} 
\usepackage{subfigure} 

\newtheorem{theo}{Theorem}[section]
\newtheorem{lemm}[theo]{Lemma}
\newtheorem{defi}[theo]{Definition}
\newtheorem{coro}[theo]{Corollary}

\newtheorem{rema}[theo]{Remark}
\numberwithin{equation}{section}
\allowdisplaybreaks
\begin{document}
	\title{ Existence and uniqueness of the globally conservative
		solutions for a weakly dissipative   Camassa-Holm equation in time weighted $H^1(\mathbb{R})$ space}
	\author{
		Zhiying $\mbox{Meng}^1$ \footnote{email: mengzhy3@mail2.sysu.edu.cn} \quad and\quad
		Zhaoyang $\mbox{Yin}^{1,2}$ \footnote{email: mcsyzy@mail.sysu.edu.cn}\\
		$^1\mbox{Department}$ of Mathematics, Sun Yat-sen University,\\
		Guangzhou, 510275, China\\
		$^2\mbox{Faculty}$ of Information Technology,\\
		Macau University of Science and Technology, Macau, China
	}
	\date{}
	\maketitle
	\begin{abstract}
		In this paper, we  prove that the existence and uniqueness of globally weak solutions to the  Cauchy problem for the weakly dissipative  Camassa-Holm equation in time weighted $H^1$ space. First, we derive an equivalent semi-linear system by introducing some new variables, and present the globally conservative solutions of this  equation in time weighted $H^1$ space. Second, we show that the peakon solutions are conservative weak solutions in $H^1.$ Finally, given a conservative solution,  we introduce
		a set of auxiliary variables tailored to this particular solution, and  prove that these
		variables satisfy a particular semilinear system having unique solutions. In turn, we get
		the uniqueness of the conservative solution in t he original variables.
	\end{abstract}
	
	\tableofcontents
	\section{Introduction}
	\par
	$~~~~$Recently, Freire  studied the weakly dissipative  Camassa-Holm  equation   \cite{Igor2020jde}
	\begin{equation}\label{u}
		\left\{\begin{array}{l}
			u_t-u_{txx}+3uu_x+\lambda(u-u_{xx})=2uu_x+uu_{xxx}+\alpha u+\beta u^2u_x+\gamma u^3u_x+\Gamma u_{xxx},~x\in\mathbb{R},\ t>0, \\
			u(0,x)=u_0,
		\end{array}\right.
	\end{equation}
	where $\alpha, \beta,\gamma, \Gamma$  are any real numbers, and $\lambda>0$. 
	The above equation \eqref{u} can be rewritten as
	\begin{equation}\label{001}
		\left\{\begin{array}{l}
			u_t+(u+\Gamma)u_x+\lambda u=Q,~x\in\mathbb{R},\ t>0, \\
			u(0,x)=u_0,
		\end{array}\right.
	\end{equation}
	with $Q=\Lambda ^{-2}\partial_x\Big(h(u)-u^2-\frac{1}{2}u_x^2\Big)$ and $h(u)=(\alpha+\Gamma)u+\frac{\beta}{3}u^3+\frac{\gamma}{4}u^4, ~\Lambda ^{-2}=(1-\partial_{xx})^{-1}.$
	The local well-posedness of its Cauchy problem in Sobolev spaces $H^s$ with $s>\frac 3 2.$ Meng and Yin \cite{MengYin} proved the local well-posedness and global strong solutions under the condition that small initial data to  \eqref{u}  in critial Besov spaces $B^s_{p,r}$ with $(\rm i).~s>1+\frac 1 p;~(\rm ii).
	~s=1+\frac 1 p,~r=1,~p\in[1,\infty).$  The integrability and the existence of global strong  solutions were studied in Sobolev spaces \cite{Silva2020}.	In particular,  the equation \eqref{u} has the portery with $\|u\|_{H^1}=e^{-\lambda t} \|u_0\|_{H^1}.$
	
	As  $\lambda=\alpha=\beta=\gamma=\Gamma=0,$ it reduces to the Camassa-Holm (CH) equation
	\cite{Constantin09,Constantin01scat}
	\begin{align}\label{CH}
		u_t-u_{xxt}=3uu_x-2u_xu_{xx}-uu_{xxx},
	\end{align}
	which is completely integrable, and has  bi-Hamiltonian structure \cite{Camassa1993,Constantin01scat}.The local well-posedness	for the Cauchy problem of the CH equation in Sobolev  spaces and Besov spaces were presented in \cite{LiA2000jde,R2001,Constantin1998,Danchin2003wp,Danchin2001inte,Li2016nwpC,YE2021A}.The ill-posedness for the CH equation has been studied in \cite{Guoyy2021Ill,Guo2019ill,Lij022Ill}.  Its existence and uniqueness of global weak solutions with initial data $u_0\in H^1(\mathbb{R})$ were proved in \cite{Constantin2000gw,Xin2000ws,Holden07lag,Bressan2015character}.~Moreover, the CH equation has globally conservative, dissipative solutions   and algebro-geometric solutions \cite{Bressan2007gd,Bressan2006g-c,Qiao2003}.
	
	In this paper, we will prove the existence and uniqueness of the globally conservative
	solutions to  \eqref{u} in time weighted $H^1$ space.  Letting $k=e^{\lambda t}u$, we conclude that $\|k\|_{H^1}= \|u_0\|_{H^1}.$ Hence, the existence and uniqueness of the globally conservative
	solutions of  \eqref{001}  in time weighted $H^1$ space can be transformed into the  existence and uniqueness of the globally conservative
	solutions to the following equation
	\begin{equation}
		\left\{\begin{aligned}
			&k_t+(e^{{-\lambda }t}k+\Gamma)k_x=-\Lambda^{-2}\partial_x(-H(k)+e^{{-\lambda }t}k^2+ \frac  {e^{{-\lambda }t}} {2}k_x^2 ),\\
			&k(0,x)=\bar{k}=u_0,\label{k11}
		\end{aligned}\right.
	\end{equation}
	where  $H(k)=(\alpha+\Gamma)k+\frac{\beta e^{{-2\lambda }t}}{3}k^3+\frac{\gamma e^{{-3\lambda }t}}{4}k^4.$ Noticing that \eqref{u} and \eqref{k11} are equivalent.  Consequently, in this paper, we  mainly study the global conservative  weak solutions of  \eqref{k11} with initial data $\bar{k}\in H^1(\mathbb{R}),$  and prove that the  equation \eqref{k11}  has a unique solution, globally in time. However, in the process of proving the globally conservative solutions, in order to get the estimate $\|\Lambda^{-2}\partial_x \cdot  \Big((\alpha+\Gamma)k\Big)\|_{L^{\infty}}$ (see Theorem \ref{global} ), we use variable transformations to handle the  term, which the idea comes from the previous works \cite{LuoznMOCH}.
	
	The paper is organized as follows. In Section 2, we give some  definitions and  estimates, which  will be used in the sequel.  {Sections 3, 4} are devoted to construct  a solution to an equivalent semi-linear system by introducing a set of new variables, this yields a conservative solution to the equation \eqref{k11}. In {Section 5}, we prove that the peakon solution of  \eqref{k11}  is  conserved in $H^1.$ In Section 6, by constructing an ordinary differential system, we prove that the conservative solutions of  \eqref{k11} is unique.
	
		\section{Premiliary}
	\par
	
	~~In this section, we first recall some definitions of globally conservative  weak solutions for \eqref{k} and give some results. We study the following equation
	\begin{equation}
		\left\{\begin{aligned}
			&k_t+(e^{{-\lambda }t}k+\Gamma)k_x=-{P}_x,\quad t>0, \ x\in \mathbb{R},\\
			&k(0,x)=\bar{k}=\bar{u},\\
		\end{aligned} \right. \label{k}
	\end{equation}	
with $P$ is defined as a convolution:
\begin{align}\label{PP}
	P\triangleq \frac 1 2  e^{-|x|}\ast\Big[-\Big((\alpha+\Gamma)k+\frac{\beta e^{{-2\lambda }t}}{3}k^3+\frac{\gamma e^{{-3\lambda }t}}{4}k^4\Big)+e^{-\lambda t}k^2+\frac  {e^{{-\lambda }t}} {2}k_x^2\Big].
\end{align}
For simplify the presentation, we introduce the following  notation
\begin{align*}
	H(k)=-\Big((\alpha+\Gamma)k+\frac{\beta e^{{-2\lambda }t}}{3}k^3+\frac{\gamma e^{{-3\lambda }t}}{4}k^4\Big),~~H_2(k)=e^{-\lambda t}k^2+\frac  {e^{{-\lambda }t}} {2}k_x^2.
\end{align*}
For smooth solutions, 	differentiating  \eqref{k} with respect to  $x$, we have
	\begin{align}\label{kx}
	k_{tx}+(e^{{-\lambda }t}k+\Gamma)k_x=e^{{-\lambda }t}k^2-H(k)- \frac  {e^{{-\lambda }t}} {2}k_x^2-{P}(t,x).
\end{align}
It follows from  \eqref{k}-\eqref{kx} that
	\begin{align}
		&(k^2)_t+(\frac{2e^{-\lambda t}k^{3}}{3}+{\Gamma k^2}+2k{P})_x=2k_x{P},\notag\\
		&(k^2_x)_t+\Big(e^{-\lambda t}kk^2_x+\Gamma k_x^2+(\alpha+\Gamma)k^2+\frac {e^{-2\lambda t}\beta k^4}{6}+\frac{e^{-3\lambda t}\gamma k^5}{10}-\frac{2e^{-\lambda t}k^3}{3}\Big)_x=-2k_x{P}.\label{k2}
	\end{align}
Hence
	\begin{align}
	{E}(t)=\Big(\int_{\mathbb{R}} (k^2+k^2_x)(t,x)dx\Big)^{\frac 1 2}=\Big(\int_{\mathbb{R}} (\bar{k}^2+\bar{k}_x^2)(x)dx\Big)^{\frac 1 2}=	{E}_0.\label{E0}
\end{align}
Setting $w=k_x^2,$ the  equation \eqref{k2} yields
\begin{align}\label{w}
	w_t+((e^{{-\lambda }t}k+\Gamma)w)_x=2k_x(e^{{-\lambda }t}k^2-H(k)-P).
\end{align}
For $\bar{k}\in H^1(\mathbb{R}),$ Young's inequality entails that
\begin{align*}
	\|{P}\|_{L^{\infty}}, ~\|{P}_x\|_{L^{\infty}}&\leq C\Big(\|e^{-|x|}\|_{L^1}\|H(k)\|_{L^{\infty}}+\|e^{-|x|}\|_{L^{\infty}}\|H_2(k)\|_{L^{1}}\Big)\notag\\
		&\leq C \Big({E}_0^{\frac 1 2}+{E}_0+ {E}_0^{\frac 3 2}+{E}_0^2\Big),\\
	\|{P}\|_{L^{2}}, ~~~\|{P}_x\|_{L^{2}}&\leq C\Big(\|e^{-|x|}\|_{L^1}\|H(k)\|_{L^{2}}+\|e^{-|x|}\|_{L^{2}}\|H_2(k)\|_{L^{1}}\Big) \notag\\& \leq C\Big({E}_0^{\frac 1 2}+{E}_0+ {E}_0^{\frac 3 2}+{E}_0^2\Big).
\end{align*}
Let us briefly recall the definition of conservative  weak solutions for convenience.
\begin{defi}\label{def1}
	Let $\bar{k}\in H^1(\mathbb{R}).$ Then $k(t,x)\in L^{\infty}(\mathbb{R}^+;H^1(\mathbb{R}) )$ is a conservative weak solution to the Cauchy problem \eqref{k} when $k(t,x)$ satisfies the following equation
	\begin{align}\label{10}
		\int_{\mathbb{R}^+}\int_{\mathbb{R}}\Big(k\psi_t+(\frac{e^{-\lambda t}k^2 }{2}+\Gamma k)\psi_x +{P}_x\psi\Big)(t,x)dxdt+\int_{\mathbb{R}} \bar{k}(x)\psi(0,x)dx=0
	\end{align}	
	for any $\psi\in C_c^{\infty}(\mathbb{R}^+,\mathcal{D}).$
	Moreover, the quantities $\|k\|_{H^1}$ are conserved in time.
\end{defi}
	\begin{defi}\label{def2}
	Let $\bar{k}\in H^1(\mathbb{R}).$ If  $k(t,x)$ is a  conservative weak solution for the Cauchy problem \eqref{k}, such that  the following properties hold:
	
	$(1).$ The funtion $k$ provides a solution for the Cauchy problem \eqref{k} in the sence  of Definition \ref{def1}.
	
	$(2).$ If $w=k_x^2$ provides a distributional solution to the balance law \eqref{w},
	\begin{align}\label{0.6}
		\int_{\mathbb{R}^+} \int_{\mathbb{R}} [k_x^2\phi_t+(e^{{-\lambda }t}k+\Gamma)w\phi_x+2k_x(e^{{-\lambda }t}k^2-H(k)-P)\phi]dxdt+\int_{\mathbb{R}}\bar{k}_x^2(x)\phi(0,x))dxdt=0,
	\end{align}
	for any test function $\phi\in C_c^1({\mathbb{R}}^2).$
\end{defi}
The main theorem of this paper is as follows. 
\begin{theo}\label{th2.3}
	For any initial data $\bar{k}\in H^1(\mathbb{R})$, the Cauchy problem \eqref{k} has a unique global conservative solution in the sense of Definition \ref{def2}.
\end{theo}
\textbf{Notations.} In Section 6, in order not to be ambiguous, we assume that $k(t,-\infty)=0,$ it follows that  $\int_{-\infty}^{y(t)}k_xdx=k(t,y(t)).$
\section{Global solutions in Lagrange coordinates}
\par

$~~$This section is devoted to getting a system equivalent to \eqref{k} by introducing a coordinate transformation into Lagrange coordinates, and to proving the existence of globally conservative solutions.
\subsection{An equivalent system}
\par
$~~$Given $\bar{k}\in H^1(\mathbb{R})$ be the initial data and   a new variable $\xi\in \mathbb{R}$. Define the nondecreasing map $\xi\mapsto \bar{y}(t,\xi)$ via the following equation
\begin{align}\label{y0}
	\int_0^{\bar{y}(\xi)}\bar{k}_x^2dx+\bar{y}=\xi.
\end{align}
Let $k=k(t,x)\in H^1(\mathbb{R})$  be the solution of equation \eqref{k} and the characteristic $y(t,\xi): t\mapsto y(t,\cdot)$  as the solutions of
	\begin{equation}
	\left\{\begin{aligned}
		&y_t(t,\xi)=e^{-\lambda  t}k(t,y(t))+\Gamma,\\
		&y(0,\xi)=\bar{y}.	
	\end{aligned} \right.\label{y}
\end{equation}
Our new variables are
\begin{align}\label{l}
	K(t,\xi)=k(t,y(t,\xi)),\quad V(t,\xi)=\frac{k_x^2\circ y}{1+k_x^2\circ y},\
	W(t,\xi)=\frac{k_x\circ y}{1+k_x^2\circ y},\quad Q(t,\xi)=({1+k_x^2\circ y})\cdot  y_{\xi},
\end{align}
From \eqref{y}-\eqref{l}, we deduce that
 \begin{align}\label{P}
	{P}(t,\xi)=\frac{1}{2}	\int_{-\infty}^{\infty}e^{-|y(t,\xi)-x|}\Big(-H(K)Q(1-V)+e^{{-\lambda }t}K^2Q(1-V)+ \frac  {e^{{-\lambda }t}} {2}QV\Big)(\eta)d\eta,
\end{align}
and $G(t,\xi)\triangleq P_x(t,y),$
\begin{align}\label{G}
	G(t,\xi)=-\frac{1}{2} \int_{\mathbb{R}} {\rm sgn} (y(t,\xi)-x)e^{-|y(t,\xi)-x|}\Big(-H(K)Q(1-V)+e^{{-\lambda }t}K^2Q(1-V)+ \frac  {e^{{-\lambda }t}} {2}QV\Big)(\eta)d\eta.
\end{align}
Noting that the fact the $y(t,\cdot)$ is an increasing function
and letting $x=y(t,\eta),$ we infer that
\begin{align}
	P(t,\xi)&=\frac{1}{2}\int_{\mathbb{R}} {\rm sgn} (\xi-\eta)e^{-|\int_{{\eta}}^{{\xi}}Q(1-V)(s)ds|}\Big(-H(K)Q(1-V)+e^{{-\lambda }t}K^2Q(1-V)+ \frac  {e^{{-\lambda }t}} {2}QV\Big)(\eta)d\eta,\label{P1}\\
	G(t,\xi)&=-\frac{1}{2}\int_{\mathbb{R}} {\rm sgn} (\xi-\eta)e^{-|\int_{{\eta}}^{{\xi}}Q(1-V)(s)ds|}\Big(-H(K)Q(1-V)+e^{{-\lambda }t}K^2Q(1-V)+ \frac  {e^{{-\lambda }t}} {2}QV\Big)(\eta)d\eta,\label{G1}
\end{align}
where the index  $t$ is omited. Now, giving another variable $Z(t,\xi)$  defined as $Z(t,\xi)=y(t,\xi)-\xi-\Gamma t,$ we obtain 
\begin{equation}
	\left\{\begin{aligned}
		&Z_t(t,\xi)=e^{-\lambda  t}K(t,\xi),\\
		&Z(0,\xi)=\bar{y}(\xi).	
	\end{aligned} \right.\label{Z}
\end{equation}
Hence, the derivatives of $G$ and $P$ are given by
	\begin{align}
	G_{\xi}(t,\xi)&=-e^{{-\lambda }t}K^2Q(1-V)-\frac {e^{{-\lambda }t}} {2}QV+H(K)Q(1-V)
	+{P}(1+Z_{\xi}),\label{Gx}\\
	{P}_{\xi}(t,\xi)&=G(1+Z_{\xi}).\label{PX}
\end{align}
Combining \eqref{y}-\eqref{l} and \eqref{P1}-\eqref{G1}, we obtain
	\begin{equation}
	\left\{\begin{aligned}
		y_t~&=e^{{-\lambda }t}K+\Gamma,\\
		K_t&=-G,\\
		V_t~&=2W\Big(e^{{-\lambda }t}K^2(1-V)-H(K)(1-V)-\frac {e^{{-\lambda }t}} {2}V-{P}(1-V)\Big),\\
		W_t&=(1-2V)\Big(e^{{-\lambda }t}K^2(1-V)-H(K)(1-V)-\frac {e^{{-\lambda }t}} {2}V-{P}(1-V)\Big),\\
		Q_t&=2WQ\Big(\frac{e^{{-\lambda }t}}{2}+e^{{-\lambda }t}K^2-H(K)-{P}\Big).\\
	\end{aligned} \right.\label{Ky}
\end{equation}
Differentiating \eqref{Ky} yields
	\begin{equation}
	\left\{\begin{aligned}
		y_{\xi t}~&=e^{{-\lambda }t}K_{\xi},\\
		K_	{t\xi }&=e^{{-\lambda }t}K^2Q(1-V)+\frac {e^{{-\lambda }t}} {2}QV-H(K)Q(1-V)-{P}(1+Z_{\xi}),\\
		V_t~~&=2W\Big(e^{{-\lambda }t}K^2(1-V)-H(K)(1-V)-\frac {e^{{-\lambda }t}} {2}V-{P}(1-V)\Big),\\
	  W_t~&=(1-2V)\Big(e^{{-\lambda }t}K^2(1-V)-H(K)(1-V)-\frac  {e^{{-\lambda }t}} {2}V-{P}(1-V)\Big),\\
	  Q_t~&=2WQ\Big(\frac{e^{{-\lambda }t}}{2}+e^{{-\lambda }t}K^2-H(K)-{P}\Big).\\
	\end{aligned}\right.\label{KX}
\end{equation}
\subsection{Global weak solutions of the equivalent system}
\par
$~~$This subsection is devoted to the proof of global solution of an equivalent semi-linear system \eqref{Ky}. Let $\bar{k}\in H^1(\mathbb{R}).$
By   \eqref{Z}, the  system \eqref{Ky} is equivalent to
\begin{equation}
	\left\{\begin{aligned}
		Z_t~&=e^{{-\lambda }t}K,\\
		K_t~&=-G,\\
		V_t~~&=2W\Big(e^{{-\lambda }t}K^2(1-V)-H(K)(1-V)-\frac {e^{{-\lambda }t}} {2}V-{P}(1-V)\Big),\\
		W_t~&=(1-2V)\Big(e^{{-\lambda }t}K^2(1-V)-H(K)(1-V)-\frac {e^{{-\lambda }t}} {2}V-{P}(1-V)\Big),\\
		Q_t~&=2WQ\Big(\frac{e^{{-\lambda }t}}{2}+e^{{-\lambda }t}K^2-H(K)-{P}\Big).\\
	\end{aligned} \right.\label{KZ}
\end{equation}
Moreover, we have
	\begin{equation}
	\left\{\begin{aligned}
		Z_{\xi t}&=e^{{-\lambda }t}K_{\xi},\\
		K_	{t\xi }&=e^{{-\lambda }t}K^2Q(1-V)+\frac {e^{{-\lambda }t}} {2}QV-H(K)Q(1-V)-{P}(1+Z_{\xi}),\\
		V_t~&=2W\Big(e^{{-\lambda }t}K^2(1-V)-H(K)(1-V)-\frac {e^{{-\lambda }t}} {2}V-{P}(1-V)\Big),\\
		W_t&=(1-2V)\Big(e^{{-\lambda }t}K^2(1-V)-H(K)(1-V)-\frac  {e^{{-\lambda }t}} {2}V-{P}(1-V)\Big),\\
		Q_t&=2WQ\Big(\frac{e^{{-\lambda }t}}{2}+e^{{-\lambda }t}K^2-H(K)-{P}\Big).\\
	\end{aligned}\right.\label{KzX}
\end{equation}
Hence, we get the following  initial data $(\bar{y}, \bar{K},\bar{V},\bar{W},\bar{Q})$ 
\begin{equation}
	\left\{\begin{aligned}
		\int_{0}^{\bar{y}}\bar{k}_x^2&dx+\bar{y}(\xi)=\xi,\\
		\bar{K}(\xi)&=\bar{k}\circ \bar{y}(\xi),\\
		\bar{V}(\xi)&=\frac{\bar{k}_x^2\circ \bar{y}}{1+\bar{k}_x^2\circ \bar{y}(\xi)},\\
		\bar{W}(\xi)&=\frac{\bar{k}_x\circ \bar{y}(\xi)}{1+\bar{k}_x^2\circ \bar{y}(\xi)},\\
		\bar{Q}(\xi)&=(1+\bar{k}_x^2\circ \bar{y})\bar{y}_{\xi}(\xi)=1.\\
	\end{aligned}\right.\label{K0}
\end{equation}
We will prove that the system \eqref{KZ} is  a well-posed system of oridinary differential equations in the Banach space $\Omega$ where
$$\Omega=H^1\cap W^{1,\infty}\times H^1\cap W^{1,\infty}\times  L^{2}\cap L^{\infty}\times L^{2}\cap L^{\infty}\times  L^{\infty}.$$
For any $X=(Z, K, V,W, Q)\in \Omega,$
 the norm on $\Omega$ is given by $$\|X\|_{\Omega}=\|Z\|_{{H^1}\cap W^{1,\infty}}\times\|K\|_{{H^1}\cap W^{1,\infty}} \times\|V\|_{{L^2}\cap L^{\infty}}\times\|W\|_{{L^2}\cap L^{\infty}}\times\|Q\|_{L^{\infty}}.$$
Indeed, for $\bar{k}\in H^1(\mathbb{R}),$ we can see that $\bar{X}=(\bar{Z},\bar{K},\bar{V},\bar{W},\bar{Q})\in \Omega,$  we thus get $ K\in W^{1,\infty}.$ In the process of proving existence, it is not necessary for $K\in W^{1,
\infty},$ which is used to prove uniqueness.

Before providing our main results in this paper, we first give the following lemmas.
\begin{lemm}\label{PG}
		Let $X=(Z, K, V, W, Q)\in \Omega,$ we define the maps $P$ and $G$ as ${P}(X):=P$ and $G(X):={P}_x\circ y$ where $P$ and $G$ are given by \eqref{P}-\eqref{G}. Then, ${P}$ and $G$ are Lipschitz maps on bounds sets from $\Omega$ to  $H^1\cap W^{1,\infty}$. Moreover, \eqref{Gx}-\eqref{PX}  hold.
\end{lemm}
\begin{proof}
	Let $\Omega_M$ is a bounded subsets of $\Omega,$ which is defined  as  $$\Omega_M=\{X=(Z, K,V, W, Q)\in \Omega| ~\|X\|_{\Omega}\leq M\}.$$
	~~\textbf{Step 1:}  ${P}$ and $G$ are maps from $\Omega_{M}$ to $H^1\cap W^{1,\infty}$.
	Combining \eqref{P}-\eqref{G} with Young's inequality, we have
		\begin{align*}
		\|{P}(X)\|_{L^{2}}, ~\|G(X)\|_{ L^{2}}&\leq C\|e^{-|x|}\|_{L^1}\Big(\|H(K)\|_{L^{2}\cap L^{\infty}}\|Q(1-V)\|_{L^{\infty}}+e^{{-\lambda }t}\|K^2\|_{L^{2}}\|Q(1-V)\|_{L^{\infty}}\notag\\
		&~~+e^{{-\lambda }t}\|Q\|_{L^{\infty}}\|V\|_{L^{2}}\Big)\leq CM,\\
			\|{P}(X)\|_{L^{\infty}}, ~\|G(X)\|_{ L^{\infty}}&\leq C\|e^{-|x|}\|_{L^1}\Big(\|H(K)\|_{L^{\infty}}\|Q(1-V)\|_{L^{\infty}}+e^{{-\lambda }t}\|K^2\|_{L^{\infty}}\|Q(1-V)\|_{L^{\infty}}\notag\\
		&~~+e^{{-\lambda }t}\|Q\|_{L^{\infty}}\|V\|_{L^{\infty}}\Big)\leq CM.
	\end{align*}
	Similarly, we obtain
	\begin{align*}
		\|P_{\xi}(X)\|_{L^{2}\cap L^{\infty}},~\|G_{\xi}(X)\|_{L^{2}\cap L^{\infty}}\leq CM.
	\end{align*}
~~\textbf{Step 2:}  $P$ and $G$ are Lipschitz maps from $\Omega_{M}$ to $H^1\cap W^{1,\infty}$.
	For $X=(Z, K, W, Q, V)$ and $\tilde{X}=(\tilde{Z}, \tilde{K}, \tilde{V}, \tilde{W}, \tilde{Q})$ be two elements in $\Omega_{M}.$  According to \cite{YE2021A}, we deduce that
\begin{align*}
		\|P(X)-P(\tilde{X})\|_{H^1\cap W^{1,\infty}},~\|G(X)-G(\tilde{X})\|_{H^1\cap W^{1,\infty}}\leq C\|X-\tilde{X}\|_{\Omega}.
\end{align*}
Combining Step 1 and Step 2, we  finish the proof of Lemma \ref{PG}.
\end{proof}
\begin{lemm}\label{lemma2}
	Let $X=(Z, K, V, W, Q)\in \Omega$ be a solution the system \eqref{KZ}. Then, for almost everywhere $\xi\in\mathbb{R},$ we have
	\begin{align}
		&W^2+V^2=V,\label{W}\\
		&y_{\xi}~=Q(1-V),\label{Q}\\
		&K_{\xi}=WQ.\label{K1}
\end{align}
Moreover, we have
\begin{align}\label{new conser}
	\frac{d}{dt}\tilde{E}(t) =\frac{d}{dt}\int_{\mathbb{R}}(K^2Q(1-V)+QV)(t,\xi)d\xi=0.
\end{align}
\end{lemm}
\begin{proof}
	Combining \eqref{Ky}-\eqref{KX}, we deduce that
	\begin{align*}
		(W^2+V^2)_t=V_t,
	\end{align*}
from which it follows
	\begin{align}\label{w1}
	(W^2+V^2)_t=&2WW_t+2VV_t\notag\\
	=&2W(2V-1)\Big(e^{{-\lambda }t}K^2(1-V)-H(K)(1-V)-\frac {e^{{-\lambda }t}} {2}V-{P}(1-V)\Big) \notag\\
	&+4VW\Big(e^{{-\lambda }t}K^2(1-V)-H(K)(1-V)-\frac {e^{{-\lambda }t}} {2}V-{P}(1-V)\Big)=V_t.
\end{align}
By the same token, we deduce that
	\begin{align}
	&y_{\xi t}=(Q(1-V))_t,\label{w2}\\
	&K_{\xi t}=(WQ)_t.\label{w3}
\end{align}
As $t=0,$   \eqref{w1}-\eqref{w3} remains hold.  Hence, we arrive at \eqref{W}-\eqref{K1}. Moreover, it follows  from  \eqref{W}-\eqref{K1} that
\begin{align}\label{VW}
	0\leq V\leq 1,~~~|W|\leq \frac 1 2.
\end{align}
Thereby, $V(t,\xi)$ and $W(t,\xi)$  are uniformly bounded in $L^{\infty}([0,T];\mathbb{R}).$

We now turn to prove \eqref{new conser}. Using the fact that $\int_{\mathbb{R}}(QV)_t(t,\xi)d\xi=2\int_{\mathbb{R}}WQP(t,\xi)d\xi=-2\int_{\mathbb{R}}2KQ(1-V)G(t,\xi)d\xi$ and \eqref{KZ}, we deduce that
\begin{align*}
\frac{d}{dt}\tilde{E}(t)
=&\int_{\mathbb{R}}\Big(2K_tKQ(1-V)+K^2Q_t-K^2(QV)_t+(QV)_t\Big)(t,\xi)d\xi\notag\\
=&\int_{\mathbb{R}}\Big(GQ(1-V)+2KWQP-2KWQP\Big)(t,\xi)d\xi\notag\\
=&\int_{\mathbb{R}}G\cdot y_{\xi}(t,\xi)d\xi=\int_{\mathbb{R}}P_{\xi}(t,\xi)d\xi=0.
\end{align*}
This means 
\begin{align}\label{lar conser}
	\tilde{E}(t)
	=\int K^2Q(1-V)+QV)(t,\xi)d\xi =	\tilde{E}(0)\triangleq \tilde{E}_0=E^2_0.
\end{align}
 Thus, the conservative law $\eqref{E0}$ in the new variables remains constant in time.
\end{proof}
The following lemma and corollary which we have learned from \cite{LuoznMOCH} are essential.
	\begin{lemm}\cite{ZMQ}\label{lemma3}
	Assume that $g(x)$ is differentiable on a.e. $[a,b],$ $f(x)\in L^1[c,d],$ and $g([a,b]) \subset [c,d].$ Then we have that $F(g(t))$ is absolutely continuous on $[a,b]$ if and only of $f(g(t))g'(t)\in L^1[c,d]$ and $\int_{g(a)}^{g(b)}f(x)dx=\int_{a}^b f(g(t))g'(t)dt$  with $F(x)=\int_c^x f(t)dt.$
\end{lemm}
\begin{coro}\cite{ZMQ}\label{coro}
	Assume that $g(x)$ is absolutely continuous on $[a,b], f(x)\in L^1[c,d],$ and $g([a,b])\subset [c,d].$ If $g(x)$ is monotonous or $f(x)\in L^{\infty}[c,d].$ Then we have $\int_{g(a)}^{g(b)}f(x)dx=\int_a^bf(g(t))g'(t)dt.$
\end{coro}
We now prove the short time existence of solutions to \eqref{KZ} as follows.
\begin{theo}\label{local existence}
		Given $\bar{k}\in H^1(\mathbb{R}).$ Then there exists a time $T>0$  such that the system \eqref{KZ}-\eqref{K0} has a unique solution   $X=(Z(t), K(t), V(t), W(t), Q(t))	\in L^{\infty}([0,T]; \Omega).$
\end{theo}
\begin{proof}
	For  $\bar{k}\in H^1(\mathbb{R}),$ one can get $\bar{X}=(\bar{Z},\bar{K},\bar{V},\bar{W},\bar{Q})\in \Omega.$
	Let $\Omega_{M}$ be a bounded subset of $\Omega,$ defined as
	$$\Omega_M=\{X=(Z, K,V, W, Q)\in \Omega| ~\|X\|_{\Omega}\leq M\}.$$
	We need to check that  the right-hand side of the system \eqref{KZ} is  Lipschitz continuous from $\Omega_{M}$ to $\Omega.$ Now, we proceed as in the proof of Lemma \ref{PG}. Therefore, the right-hand side of the system \eqref{KZ} is Lipschitz on $\Omega_M.$ By the standard theory of ordinary differential equations, we  conclude that there exists a unique solution $X=(Z(t), K(t), V(t), W(t), Q(t))$ be the short time solution of the system \ref{KZ} in $L^{\infty}([0,T]; \Omega).$
\end{proof}
Next, we turn to the proof of existence of global solutions of the Cauchy problem \eqref{KZ}-\eqref{K0}.
\begin{theo}\label{global}
Let $\bar{k}\in H^1(\mathbb{R}).$ Then the local solution $X=(Z(t), K(t), V(t),W(t), Q(t))$ of \eqref{KZ} is a unique globally conservative solution  in $L^{\infty}(\mathbb{R}^+; \Omega).$
\end{theo}
\begin{proof}
In order to prove the 	existence the global solutions, we shall demonstrate  the local solution $X=(Z(t), K(t), V(t),W(t), Q(t))$ is uniformly bounded in $\Omega$ on any bounded time interval $[0,T]$ with any $T>0.$  Lemma \ref{lemma2} guarantees that
	\begin{align}\label{supk}
	\sup\limits_{\xi\in \mathbb{R}}|K^2(\xi)|\leq2\int_{\mathbb{R}}|KK_{\xi}|d\xi\leq 2(\int_{\mathbb{R}}|K^2Q(1-V)|d\xi)^{\frac{1}{2}} (\int_{\mathbb{R}}|(QV)|d\xi)^{\frac{1}{2}} \leq C\tilde{E}(0).
\end{align}
Combining \eqref{VW} and \eqref{supk}, we infer that $K,~V$ and $W$ are uniformly bounded  in $L^{\infty}.$  However, we cannot obtain $\int_{\mathbb{R}}e^{-|\int_{
\eta}^{\xi}Q(1-V)(s)ds|}KQ(1-V)(\eta)d\eta$ is bounded in $L^{\infty}.$ Therefore, we use variable transformations and contradiction argument to handle the problem. According to the Cauchy problem \eqref{KZ}-\eqref{K0}, we get $\bar{y}(\xi)\in L^{\infty}_{loc}$ is strictly monotonous and
\begin{align*}
	|\bar{y}(\xi_2)-\bar{y}(\xi_1)|=|\int_{\bar{y}(\xi_1)}^{\bar{y}(\xi_2)}1~dx|\leq|\int_{\bar{y}(\xi_1)}^{\bar{y}(\xi_2)}(1+\bar{k}_x^2)dx|\leq|\xi_2-\xi_1|,
\end{align*}
from which implies $\bar{y}(\xi)$ is local Lipschitz continuous function. Lemma \ref{local existence} entails that $K(t,\xi)$ is Lipschitz continuous as it maps $\Omega_{M}$ to $H^1\cap W^{1,\infty}.$ From \eqref{KZ}-\eqref{K0}, there exists a $0\leq T<\infty$ such that $y(t,\xi)\in H^1_{loc}$ for $t\in [0,T),$ which means $y(t,\xi)$ is a local absolutely continuous function for  $t\in [0,T).$ Making use of Corollary \ref{coro} for $t\in [0,T)$ and $[a,b]\subset \mathbb{R},$ we arrive at
\begin{align}
		\|\frac 1 2 \int_a^b e^{-|\int_{\eta}^{\xi}Q(1-V)(s)ds|} |K|Q(1-V)(\eta)d\eta \|_{L^{\infty}}&\leq \frac 1 2 \|K\|_{L^{\infty}}\|\int_a^b e^{-|y(\xi)-y(\eta)|} y_{\xi}d\eta\||_{L^{\infty}}\notag\\ 
	&\leq \frac 1 2 \|K\|_{L^{\infty}}\int_{-\infty}^{+\infty}e^{-|y|}dy.\label{y1}
\end{align}
	As $a\rightarrow -\infty, b\rightarrow +\infty,$ the left side of \eqref{y1} is monotonic. Applying the monotonic convergence theorem, we see  that there exists a limit on the left side of \eqref{y1}. Therefore, we obtain
	\begin{align*}
	\|\int_{\mathbb{R}}e^{-|\int_{
\eta}^{\xi}Q(1-V)(s)ds|}KQ(1-V)(\eta)d\eta\|_{L^{\infty}}\leq\frac 1 2\|K\|_{L^{\infty}}\int_{-\infty}^{+\infty}e^{-|s|}ds\leq C\tilde{E}_0^{\frac 1 2}.	
	\end{align*}
Combining the above estimate and \eqref{VW}-\eqref{supk}, we have
\begin{align}\label{Gsup}
	\|G\|_{L^{\infty}}&\leq C\Big(\tilde{E}(0)^{\frac 1 2}+\|e^{-|x|}\|_{L^{\infty}}(\int K^2Q(1-V)+QVd\eta)\notag\\&~~+\|e^{-|x|}\|_{L^{\infty}}\cdot(\|K\|_{L^{\infty}}+\|K\|^2_{L^{\infty}})\int K^2Q(1-V)d\eta \Big)\notag\\
	&\leq C\Big({\tilde{E}_0}^{\frac 1 2}+{\tilde{E}_0}+\tilde{E}_0^{\frac 3 2}+{\tilde{E}^2_0}\Big).
\end{align}
Likewise, we get
\begin{align}\label{pp}
	\|(P,P_{\xi})\|_{L^2\cap L^{\infty}},~\|(G,G_{\xi})\|_{L^2\cap L^{\infty}}\leq C\Big({\tilde{E}_0}^{\frac 1 2}+{\tilde{E}_0}+\tilde{E}_0^{\frac 3 2}+{\tilde{E}^2_0}\Big).
\end{align}
Hence, one can get from \eqref{KZ} that
\begin{align*}
	|Q_t|\leq  C\Big({\tilde{E}_0}^{\frac 1 2}+{\tilde{E}_0}+\tilde{E}_0^{\frac 3 2}+{\tilde{E}^2_0}\Big),
\end{align*}
 which implies
\begin{align*}
		\exp\{-C\Big({\tilde{E}_0}^{\frac 1 2}+{\tilde{E}_0}+\tilde{E}_0^{\frac 3 2}+{\tilde{E}^2_0}\Big)t\}\leq Q(t)\leq 	\exp\{{C\Big({\tilde{E}_0}^{\frac 1 2}+{\tilde{E}_0}+\tilde{E}_0^{\frac 3 2}+{\tilde{E}^2_0}\Big)t}\}.
\end{align*}
From \eqref{Q}, we have
\begin{align}
\|y_{\xi}\|_{L^{\infty}}\leq\|Q\|_{L^{\infty}}\leq \exp\{{C\Big({\tilde{E}_0}^{\frac 1 2}+{\tilde{E}_0}+\tilde{E}_0^{\frac 3 2}+{\tilde{E}^2_0}\Big)T}\}.	
\end{align}
It follows that
\begin{align}\label{ysup}
	\bar{y}(\xi)-C\Big({\tilde{E}_0}^{\frac 1 2}+{\tilde{E}_0}+\tilde{E}_0^{\frac 3 2}+\tilde{E}^2_0\Big)t
\leq y(t,\xi)\leq \bar{y}(\xi)+C\Big({\tilde{E}_0}^{\frac 1 2}+{\tilde{E}_0}+\tilde{E}_0^{\frac 3 2}+\tilde{E}^2_0\Big)t.
\end{align}
This means   $y(t,\xi)$ is bounded in $L^{\infty}_{loc}$ for  all $t\in [0,T],$ then we have $y(t,\xi)\in H^1_{loc}$ for $t\in [0,T].$ According to contradiction argument, we can prove that $T$  in the above results connot have a upper bound, we conclude that the above the results are vaild for  $t\in \mathbb{R}.$

Taking advantage of \eqref{Gsup}-\eqref{ysup} and the system \eqref{KZ}, we get the following estimates
\begin{align*}
		\frac{d}{dt}\|K\|^2_{L^2} &\leq C\|K\|_{L^{2}}\Big({\tilde{E}_0}^{\frac 1 2}+{\tilde{E}_0}+\tilde{E}_0^{\frac 3 2}+\tilde{E}^2_0\Big),\\
			\frac{d}{dt}\|V\|^2_{L^{2}} &\leq C(\|V\|^2_{L^{2}}+\|V\|_{L^{2}}) \Big({\tilde{E}_0}^{\frac 1 2}+{\tilde{E}_0}+\tilde{E}_0^{\frac 3 2}+\tilde{E}^2_0\Big),
\end{align*}
and
\begin{align*}
		\frac{d}{dt}\|W\|^2_{L^{2}}&\leq C \|W\|_{L^{2}}\Big({\tilde{E}_0}^{\frac 1 2}+{\tilde{E}_0}+\tilde{E}_0^{\frac 3 2}+\tilde{E}^2_0\Big).~~~~~~~~~~~~
\end{align*}
In addition, using \eqref{W}-\eqref{Q}, we have
\begin{align*}
	\|K_{\xi}\|_{L^{2}}\leq C\|W\|_{L^2}\|Q\|_{L^{\infty}},~~	\|K_{\xi}\|_{L^{\infty}}\leq C\|W\|_{L^{\infty}}\|Q\|_{L^{\infty}}.
\end{align*}
Therefore, we finish the proof of Theorem \ref{global}.
\end{proof}
\begin{theo}\cite{Holden07lag}\label{meas}
	Let  $X=(Z,K,V,W,Q)$ be the  corresponding solution  of the system \eqref{KZ} with the initial data $\bar{X}=(\bar{Z},\bar{K},\bar{V},\bar{W},\bar{Q})\in L^{\infty}([0,T];\Omega)$ given by Theorem \ref{local existence}. Then, $(Z_{\xi},K_{\xi},V,W,Q)$ is a solution of the system \eqref{KzX}, and
	$$(Z_{\xi},K_{\xi},V,W,Q)\in (L^2\cap L^{\infty})^4\times L^{\infty}.$$ Moreover, we have $y_{\xi}\geq 0$ and $meas\mathcal{(N)}=0$ with
	$$\mathcal{N}=\{(t,\xi)\in[0,T]\times \mathbb{R}~|~y_{\xi}=0\}.$$
\end{theo}
\begin{rema}\label{global Ky}
	Let $\bar{k}\in H^1(\mathbb{R}),$ the system \eqref{Ky} also has a globally unique solution in $L^{\infty}(\mathbb{R}^+; \Omega).$
\end{rema}
	\section{Solutions to the original equation}
\par
~~This section is devoted to proving that the globally conservative weak solution to the original equation.
	\begin{theo}\label{global k}
	Let $\bar{k}\in H^1(\mathbb{R}).$ Then, the Cauchy problem \eqref{k} has a globally conservative solution in the sence of Definition \ref{def1}.
\end{theo}
\begin{proof}
	From Remark \ref{global Ky}, we know that the system \eqref{Ky} has a unique globally conservative weak solution. Then, the mapping $t\mapsto y(t,\xi)$ provides a solution to the Cauchy problem
	\begin{equation}
		\left\{\begin{aligned}
			&y_t(t,\xi)=e^{-\lambda  t}K(t,\xi)+\Gamma,\\
			&y(0,\xi)=\bar{y}(\xi).	
		\end{aligned} \right.\label{oy}
	\end{equation}
Set
	\begin{align}\label{kK}
	k(t,x)=K(t,\xi),~~if~x=y(t,\xi).
\end{align}	
	We need to check that \eqref{kK} is   well-defined. \eqref{y0} and \eqref{ysup} entail that
	\begin{align*}
		\lim_{\xi\rightarrow\pm\infty}y(t,\xi)=\pm\infty.
	\end{align*}
Thanks to \eqref{Q}, we see that $y_{\xi}\geq 0$ for all $t\geq 0$ and $a.e.~\xi\in \mathbb{R}.$ Moreover, the map $\xi\mapsto y(t,\xi)$ is nondecreasing. Assume that  ${\xi}_1<\xi_2$ but $y(t,\xi_1)=y(t,\xi_2),$ it follows that
	$$0=\int_{\xi_1}^{\xi_2}y_{\xi}(t,\eta)d\eta=\int_{\xi_1}^{\xi_2}Q(1-V)(t,\eta)d\eta.$$
	If $Q\neq0$, we deduce that  $V=1$ and   $W=0$ in $[\xi_1,\xi_2].$ Therefore, we have
		$$K(t,\xi_1)-K(t,\xi_2)=\int_{\xi_1}^{\xi_2}K_{\xi}(\eta)d\eta=\int_{\xi_1}^{\xi_2}WQ(\eta)d\eta.$$
	If $Q=0,$  the above equality also make sense. Then,  for all $t\geq0$ and $x\in\mathbb{R}, $ the map $(t,x)\mapsto k(t,x)$ is well-defined. According to the definition \eqref{kK},  we give
		\begin{align}\label{kx1}
		k_x(t,y(t,\xi))=\frac{W}{1-V},~if~x=y(t,\xi),~ y_{\xi}\neq0.
	\end{align}
Applying \eqref{lar conser}, \eqref{kx1} and changing the variables, we see that
\begin{align*}
	E(t)=&\int_{\mathbb{R}}(k^2+k_x^2)(x)dx=\int_{\mathbb{R}\cap\{y_{\xi}\neq0\}}(k^2+k_x^2)(t,y(t,\xi))y_{\xi}d\xi\notag\\
	=&\int_{\mathbb{R}\cap\{y_{\xi}\neq0\}}(K^2Q(1-V)+QV)(\xi)d\xi=\int_{\mathbb{R}}(K^2Q(1-V)+QV)(\xi)d\xi\notag\\
	=&\tilde{E}(t)=\tilde{E}_0=\int_{\mathbb{R}}(\bar{k}^2+\bar{k}_x^2)(x)dx,
\end{align*}
which implies  $k$ is uniformly bounded in $H^1(\mathbb{R}).$ On the other hand, $k$ satisfies  \eqref{k}. Indeed, for every $\phi\in C^{\infty}_c(\mathbb{R}^+,\mathcal{D}),$ we have
	\begin{align*}
	&\int_{\mathbb{R}^+}\int_{\mathbb{R}}(-k\phi_t+(e^{{-\lambda }t}k+\Gamma)(t,x)k_x\phi(t,x)dxdt\notag\\
	=&\int_{\mathbb{R}^+}\int_{\mathbb{R}}(-k\phi_t+(e^{{-\lambda }t}k+\Gamma)k_x\phi)(t,y(t,\xi))y_{\xi}d\xi dt\notag\\
	=&\int_{\mathbb{R}^+}\int_{\mathbb{R}}-K\phi_t(t,y(t,\xi)y_{\xi}+(e^{{-\lambda }t}K+\Gamma)K_{\xi}\phi(t,y(t,\xi))d\xi dt	\notag\\
	=&\int_{\mathbb{R}^+}\int_{\mathbb{R}}-K(\phi(t,y(t,\xi)y_{\xi})_t+e^{{-\lambda }t}(K^2\phi (t,y(t,\xi)))_{\xi}+\Gamma(K\phi(t,y(t,\xi)))_{\xi}d\xi dt\notag\\
	=&\int_{\mathbb{R}^+}\int_{\mathbb{R}}-U(\psi(t,y(t,\xi)y_{\xi})_td\xi dt\notag\\
	=&\int_{\mathbb{R}^+}\int_{\mathbb{R}}K_t\phi(t,y(t,\xi))y_{\xi}d\xi dt+\int_{\mathbb{R}}\bar{K}(\xi)\phi(0,\xi)\bar{y}_{\xi}(\xi)d\xi\notag\\
	=&\int_{\mathbb{R}^+}\int_{\mathbb{R}}-P_x(t,y(t,\xi))y_{\xi}d\xi dt+\int_{\mathbb{R}}\bar{K}(\xi)\phi(0,\xi)\bar{y}_{\xi}d\xi\notag\\
	=&\int_{\mathbb{R}^+}\int_{\mathbb{R}}-P_x(t,x)dxdt+\int_{\mathbb{R}}\bar{k}(x)\phi(0,x)dx,
	\end {align*}
		with
	$$(\phi(t,y(t,\xi)y_{\xi})_t=\phi_t(t,y(t,\xi))\cdot y_{\xi}+\phi_x((t,y(t,\xi))(K+\Gamma)(t,\xi)\cdot y_{\xi}+\phi  (t,y(t,\xi)) (K+\Gamma)_{\xi},$$ and
	\begin{align*}
		P_x(t,y(t,\xi))=&\frac{1}{2}(\int_{y(t,\xi)}^{+\infty}-\int_{-\infty}^{y(t,\xi)})e^{-|y(t,\xi)-x|}(-H(k)+e^{{-\lambda }t}k^2+ \frac  {e^{{-\lambda }t}} {2}k_x^2)(t,x)dx\notag\\
		=&\frac{1}{2}(\int_{\xi}^{+\infty}-\int_{-\infty}^{{\xi}})e^{-|\int_{\eta}^{{\xi}}Q(1-V)(t,s)ds|}(-H(K)Q(1-V)+e^{{-\lambda }t}K^2Q(1-V)+\frac  {e^{{-\lambda }t}} {2}QV)d\xi'\notag\\
		=&G(t,\xi).
	\end{align*}
Likewise, we get
	\begin{align}\label{ww}
	\int_{\mathbb{R}^+} \int_{\mathbb{R}} [k_x^2\phi_t+(e^{{-\lambda }t}k+\Gamma)w\phi_x+2k_x(e^{{-\lambda }t}k^2-H(k)-P)\phi]dxdt+\int_{\mathbb{R}}\bar{k}_x^2(x)\phi(0,x))dxdt=0.
\end{align}
Hence, we conclude that $k(t,x)$ is a globally conservative solution to   \eqref{k}  in the sense of Definition \ref{def2}.
\end{proof}
	\section{Peakon solutions}
\par
\quad \quad In this section, we give a conservative  solution in time-weighted $H^1$  space with the peakon solutions of \eqref{u} in the following form
\begin{align}\label{p1}
	u(t,x)=\sum_{i=1}^n p_i (t)e^{-|x-q_i (t)|},
\end{align}
where $p_i (t),q_i (t), i=1,..., n$ are smooth functions with respect to $t.$
\begin{theo}
	Let $q_1(t)<q_2(t)<...<q_n(t).$ Then \eqref{p1}  are weak solutions of O.D.E. in the following 
	\begin{equation*}
		\left\{\begin{aligned}
			&2\dot{p}_i-2p_i(a_i-b_i)+2\lambda p_i=0,\\
			&-2p_i\dot{q}_i+2p_i(a_i+b_i+p_i)+2\Gamma p_i=0,\\
		\end{aligned}\right.
	\end{equation*}
	where $a_i =\sum_{j<i} p_j e^{q_j-q_i}, b_i=\sum_{j<i} p_j e^{q_i-q_j}.$
\end{theo}
\begin{proof}
	For any $i\in \{0,...n+1\},$ let
	\begin{align*}
		u_i(x,t)=\sum_{j=1}^{i} p_j(t) e^{q_j(t)-x}+\sum_{j=i+1}^n p_j(t) e^{x-q_j(t)},
	\end{align*}
	where $u_i(t,x)\in C^{\infty}$ in the space variable. Then \eqref{p1} can be rewritten as
	\begin{align*}
		u(t,x)=\sum_{i=0}^{n}u_i(t,x)\chi_i(x),
	\end{align*}
	which $\chi_i$ represents the characteristic function in interval $[q_i, q_{i+1}),\ i=1,...,n$ and $q_0=-\infty,\ q_{n+1}=\infty.$ Owing that $\chi_i$ has disjoint supports, we have
	\begin{align*}
		(u+\Gamma)u_x=\sum_{i=0}^n (u_i+\Gamma) u_{i,x}\chi_i.
	\end{align*}
Therefore
	\begin{align}\label{p6}
		\Big((u+\Gamma)u_x\Big)_x&=\sum_{i=0}^n\Big((u_i+\Gamma) u_{i,x}\Big)_x \chi_i+\sum_{i=1}^n((u_i+\Gamma) u_{i,x})(q_i)\delta_{q_i}-\sum_{i=0}^{n-1}\Big(((u_{i}+\Gamma) u_{{i},x}\Big)(q_{i+1})\delta_{q_{i+1}}\notag\\
		&=\sum_{i=0}^n\Big((u_i+\Gamma) u_{i,x}\Big)_x \chi_i+\sum_{i=1}^n[(u_i+\Gamma)u_{i,x}]_{q_i} \delta_{q_i},
	\end{align}
	which $[v]_{{q_i}}=v(q_i^+)-v(q_i^-).$
 Noting that $u$ is continuous, one can get $u=u_i=u_{i,xx}$ on every  interval $(q_i,q_{i+1})$ and $[u^2]_{q_i}=0.$   Differentiating \eqref{p6}, we obtain
\begin{align}\label{p7}
	\Big((u+\Gamma)u_x\Big)_{xx}&=\sum_{i=0}^n\Big((u_i+\Gamma) u_{i,x}\Big)_{xx} \chi_i+\sum_{i=1}^n[((u_i+\Gamma)u_{i,x})_x]_{q_i} \delta_{q_i}+\sum_{i=1}^n[(u_i+\Gamma)u_{i,x}]_{q_i} \delta_{q_i}'\notag\\
	&=\sum_{i=0}^n\Big((u_i+\Gamma) u_{i,x}\Big)_{xx} \chi_i+\sum_{i=1}^n[u_x^2+u^2+\Gamma u]_{q_i} \delta_{q_i}+	\sum_{i=1}^n[(u+\Gamma)u_{x}]_{q_i} \delta_{q_i}'\notag\\
	&=\sum_{i=0}^n\Big((u_i+\Gamma) u_{i,x}\Big)_{xx} \chi_i+\sum_{i=1}^n[u_x^2+\Gamma u]_{q_i} \delta_{q_i}+	\sum_{i=1}^n[(u+\Gamma)u_{x}]_{q_i} \delta_{q_i}'.
\end{align}

Likewise, we have
\begin{align}
	(\frac1 2 u_x^2-\alpha u-\frac{\beta} {3}u^3-\frac{\gamma}{4}u^4)_x&=\sum_{i=0}^n 	(\frac1 2 u_{i,x}^2-\alpha u_i-\frac{\beta} {3}u_i^3-\frac{\gamma}{4}u_i^4)_x\chi_i\notag\\
	&~~~+\sum_{i=1}^n [\frac1 2 u_x^2-\alpha u-\frac{\beta} {3}u^3-\frac{\gamma}{4}u^4]_{q_i}\delta_{q_i},\label{p8}\\
	u_t-u_{xxt}=\sum_{i=0}^n (u_{i,t}-&u_{i,xxt})\chi_i-\sum_{i=1}^n ([u_{xt}]_{q_i}\delta_{q_i}+[u_t]_{q_i}\delta_{q_i}')\label{p9}.
\end{align}
It follows from \eqref{p6}-\eqref{p9} that
\begin{align*}
	&\sum_{i=0}^n 		\{u_{i,t}-u_{i,xxt}+3u_iu_{i.x}-\Big((u_{i}+\Gamma)u_{i,x}\Big)_{xx}+(\frac1 2 u_{i,x}^2-\alpha u_i-\frac{\beta} {3}u_i^3-\frac{\gamma}{4}u_i^4)_x+\lambda(u_i-u_{i,xx})\}\chi_i\notag\\
	&~~~+\sum_{i=1}^n\{-[u_{i,t}]_{q_i}-[u_x^2+u^2+\Gamma u]_{q_i}+ [\frac1 2 u_x^2-\alpha u-\frac{\beta} {3}u^3-\frac{\gamma}{4}u^4]_{q_i}-\lambda[u_x]_{q_i}\}\delta_{q_i}\notag\\
	&~~~+\sum_{i=1}^n\{-[u_t]_{q_i}-[(u+\Gamma)u_{x}]_{q_i}-\lambda[u]_{q_i}\}\delta_{q_i}'=0.
\end{align*}
For $\chi_i, \delta_{q_j}, \delta_{q_k}', i,j,k=1,...,n,$ we have
\begin{equation}
	\left\{\begin{aligned}
		&-[u_{i,t}]_{q_i}-[u_x^2]_{q_i}+[u^2]_{q_i}+\Gamma [u]_{q_i}+ \frac1 2  [u_x^2]_{q_i}-\alpha  [u]_{q_i}-\frac{\beta} {3} [u^3]_{q_i}-\frac{\gamma}{4} [u^4]_{q_i}-\lambda[u_x]_{q_i}=0,\\
		&-[u_t]_{q_i}-[(u+\Gamma)u_{x}]_{q_i}-\lambda[u]_{q_i}=0.\label{p11}
	\end{aligned}\right.
\end{equation}
According to the definition of $a_i,b_i,i=1,...,n,$ we have
\begin{align}
	&[u_t]_{q_i}=2p_i\dot{q}_i,~ [u_{xt}]_{q_i}=-2\dot{p}_i,~~~~~~~~~~ [uu_x]_{q_i}=-2p_i(a_i+p_i+b_i),\label{p12}\\
	&[u_x]_{q_i}=-2{p}_i,\  [u_x^2]_{q_i}~=4p_i(a_i-b_i),\ [u]_{q_i}=[u^2]_{q_i}=[u_{xx}]_{q_i}=0.\label{p13}
\end{align}

Substituting \eqref{p12}, \eqref{p13} into \eqref{p11}, we end up with
\begin{equation*}
	\left\{\begin{aligned}
		&2\dot{p}_i-2p_i(a_i-b_i)+2\lambda p_i=0,\\
		&-2p_i\dot{q}_i+2p_i(a_i+p_i+b_i)+2\Gamma{p}_i=0.
	\end{aligned}
	\right.
\end{equation*}
Now, we consider that $n=1$ and $a_1=b_1=0.$ Then we have
\begin{equation*}
	\left\{\begin{aligned}
		&2\dot{p}+2\lambda p=0,\\
		&-2p\dot{q}+2p^2+2\Gamma p=0,
	\end{aligned}\right.
\end{equation*}
which implies that
\begin{equation}
	\left\{\begin{aligned}
		&p=p(0)e^{-\lambda t},\\
		&q=\frac {1}{\lambda}p(0)(1-e^{-\lambda t})+\Gamma t+q(0), \quad\quad \lambda >0,\\
		&q=p(0)t+\Gamma t+q(0), \quad\quad \quad\quad\quad\quad\quad\lambda =0.
	\end{aligned}\right.
\end{equation}

It follows that
\begin{align*}
	u(t,x)&=p(0)e^{-\lambda t}e^{-|x-q(t)|},\\ k(t,x)&=e^{\lambda t} u(t,x)=p(0)e^{-|x-q(t)|}\Rightarrow \|k\|_{H^1}=\|\bar{k}\|_{H^1}.
\end{align*}
Figure 1 shows the evolution behavior of single peak solitary solutions with dissipative coefficient $\lambda$ with $p(0)=\frac 1 2,~q(0)=1,~\Gamma=-2,$
\begin{figure*}[h]
	\centering
	\subfigure[$\lambda=1$]{
		\includegraphics[scale=0.25]{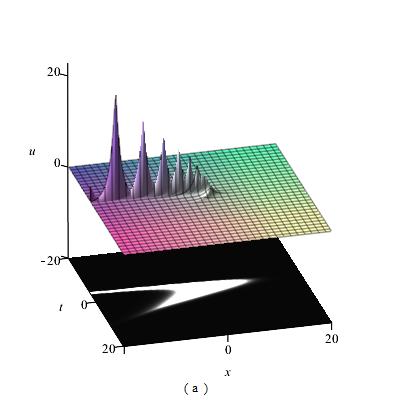}
	}
	\quad
	\subfigure[$\lambda=\frac{1}{4}$]{
		\includegraphics[scale=0.25]{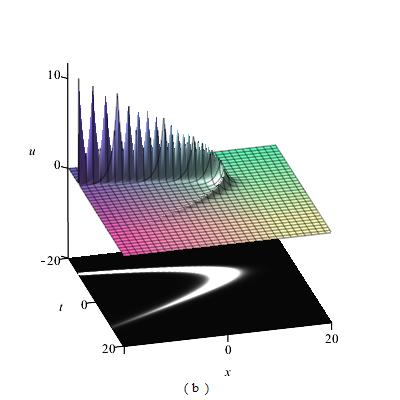}
	}
	\quad
	\subfigure[$\lambda=\frac{1}{10}$]{
		\includegraphics[scale=0.25]{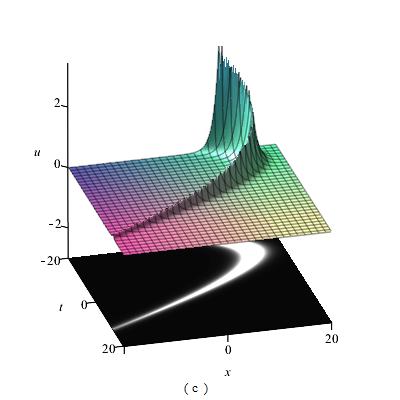}\label{c}
	}
\end{figure*}
$\\$
\begin{figure*}[h]
		\centering
	\subfigure[$\lambda=\frac{1}{15}$]{
		\includegraphics[scale=0.25]{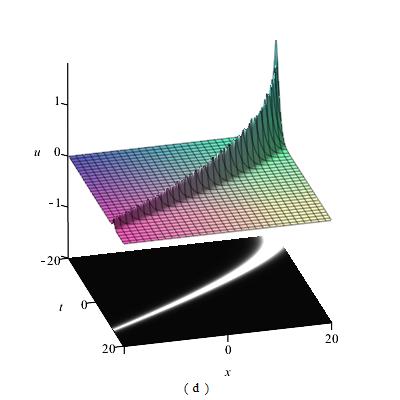}\label{d}
	}
	\subfigure[$\lambda=\frac{1}{30}$]{
		\includegraphics[scale=0.25]{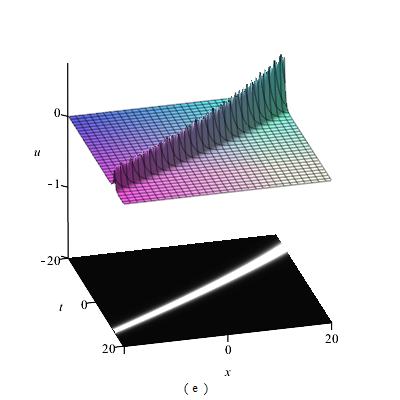}\label{e}
	}
	\quad
	\subfigure[$\lambda=0$]{
		\includegraphics[scale=0.25]{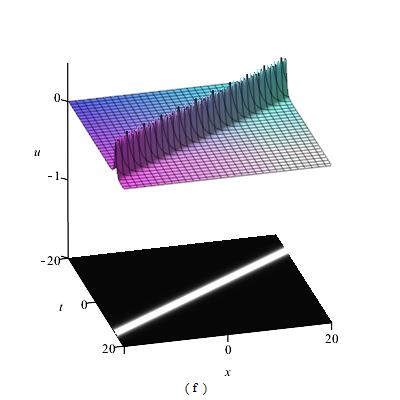}\label{f}
	}
	\caption{ (a)$\lambda=1$; (b)$\lambda=\frac{1}{4}$; (c)$\lambda=\frac{1}{10}$; (d)$\lambda=\frac{1}{15}$; (e)$\lambda=\frac{1}{30}$; (f)$\lambda=0$.}
\end{figure*}
$\\$
 Note that $(1-\partial_{xx})u=\delta_0 \in \mathcal{M} \hookrightarrow B^0_{1,\infty}\hookrightarrow B^{-\frac 1 2}_{1,\infty},$  where $\mathcal{M}$ is bounded measures spaces. Thanks to $B^{\frac 3 2}_{2,\infty} \hookrightarrow H^1(\mathbb{R})$, we get $u\in  H^1(\mathbb{R}).$ Hence, $k(t,x)$ is a conservative solution of equation \eqref{u}  in $H^1,$ in other words,  $u(t,x)$ is a conservative  solution in time weighted $H^1$ space. Moreover,  we infer that
$$\lim\limits_{\lambda \rightarrow 0} q(t)=q(0).$$
\end{proof}

\section{Uniqueness of solutions
	for the original equation}
\par

~~~In this section, we consider the uniqueness solutions  for \eqref{k}. Our main result is as follows.
\begin{theo}\label{th2.4}
	Let $k(t,x)\in H^1$ be a conservative weak solutions to the Cauchy problem \eqref{k} in the sense of Definition \ref{def2}. Then $k(t,x)$ is unique.
\end{theo}
\subsection{Uniqueness of characteristic}
~~	This subsection is devoted to the  study of the uniqueness of the characteristic to \eqref{k}. 	Let $k=k(t,x)$ be a conservative solution of  \eqref{k} and satisfy \eqref{w}. Let $y(t,\xi)$ still denote the characteristic 
	\begin{equation}
	\left\{\begin{aligned}
		&y_t(t,\xi)=e^{-\lambda  t}k(t,y(t))+\Gamma,\\
		&y(0,\xi)=\bar{y}.	
	\end{aligned} \right.\label{yy}
\end{equation}
 Introduce new coordinates $(t,\beta)$ relative to the original coordinates $(t,x)$ by the following transformation
	\begin{align}\label{yb}
	y(t,\beta)+\int_{-\infty}^{y(t,\beta)}k_x^2(t,z)dz=\beta.
\end{align}
	At time $t$ where the measure $\mu_{(t)}$ is not absolutely continuous
with respect to the Lebesgue measure,
For any time $t$ and $\beta\in\mathbb{R}$,  define $y(t, \beta)$ to be the unique $y$ such that
\begin{align}\label{yb1}
	y(t, \beta)+\mu_{(t)}\{(-\infty,y)\}\leq \beta\leq y(t, \beta)+\mu_{(t)}\{(-\infty,y]\}.
\end{align}	
Combining \eqref{y1} and \eqref{yy}, we get
\begin{align}\label{yy1}
	\frac{d}{dt}\int_{-\infty}^{y(t)}k_x^2dx=\int_{-\infty}^{y(t)}2k_x(e^{{-\lambda }t}k^2-H(k)-P)dx.
\end{align}
Now we give the  following lemma   to prove the Lipschitz continuity of $x$ and $k$ as functions of the variables $t,\beta.$
	\begin{lemm}\label{kycontiu}
	Let $k(t,x)$ be a  conservative  solution  of \eqref{k}. Then, for all $t\geq 0$, the following maps
	\begin{align*}
		&\beta \mapsto y(t,\beta),\\
		&\beta \mapsto k(t,y(t,\beta)),
	\end{align*}
	    defined by \eqref{yb1}, are Lipschitz continuous with constant $1.$  Moreover, the map $t\mapsto y(t,\beta)$ is also Lipschitz continuous with a
	constant depending only on $\|\bar{k}\|_{H^1}.$
\end{lemm}
\begin{proof}
	\textbf{Step 1.} For any fixed time $t\geq 0$, the  map
	\begin{align*}
		x\mapsto \beta(t,y):=x+\int_{-\infty}^{x}k_x^2(t,y)dy
	\end{align*}
is right continuous and strictly increasing. This means the inverse $\beta\mapsto x(t,\beta)$ is well-defined, continuous, nondecreasing.  Given  ${\beta}_1 <{\beta}_2$, we have
	\begin{align*}
		y(t,{\beta}_2)-	y(t,{\beta}_1)=({\beta}_2 -{\beta}_1)-\int_{y(t,{\beta}_1)}^{y(t,{\beta}_2)}k_x^2(t,z)dz\leq {\beta}_2 -{\beta}_1.
	\end{align*}
	One can conclude that $\beta \mapsto y(t,\beta)$ is Lipschitz continuous.
	
	\textbf{Step 2.}  Let ${\beta}_1 <{\beta}_2$. Then,  it follows that
	\begin{align*}
		|k(t,{\beta}_2)-k(t,{\beta}_1)|&\le \int_{y(t,{\beta}_1)}^{y(t,{\beta}_2)}|k_x(t,z)|dz\leq \int_{y(t,{\beta}_1)}^{y(t,{\beta}_2)}\frac{1}{2}(1+k_x^2)dz \notag\\
		&\leq \frac{1}{2} [	y(t,{\beta}_2)-	y(t,{\beta}_1)+\int_{y(t,{\beta}_1)}^{y(t,{\beta}_2)}k_x^2(t,z)]dz\leq {\beta}_2 -{\beta}_1.	
	\end{align*}
Therefore, we arrive at the map $\beta \mapsto k(t,y(t,\beta))$	is Lipschitz continuous.

\textbf{Step 3.} According to \eqref{E0}, we have
	\begin{align*}
		\|2[e^{{-\lambda }t}k^2-P-H(k)]k_x\|_{L^1}	\leq \|-H(k)+e^{{-\lambda }t}k^2-P\|_{L^2}\|k_x\|_{L^2}\leq C_{E_0}.
	\end{align*}
If $t> \tau,$ we  obtain
	\begin{align*}
		{\mu}_t\{(-\infty,y-C_{E_0}(t-\tau))\}\leq &{\mu_{\tau}}	\{(-\infty,y)\}+\int_{\tau}^t\|2[e^{{-\lambda }t}k^2-P-H(k)]k_x\|_{L^1}dt\notag\\
		\leq&{\mu_{\tau}}	\{(-\infty,y)\}+C_{E_0}(t-\tau).
	\end{align*}
	Defining $y^-(t):=y-C_{E_0}(t-\tau),$ we get
	\begin{align*}
		y^-(t)+{\mu}_t\{(-\infty,y^-(t)]\}&\leq y-C_{E_0}(t-\tau)+{\mu}_{\tau}\{(-\infty,y)\}+C_0(t-\tau)\notag\\
		&\leq y+{\mu}_{\tau}\{(-\infty,y)\}\leq \beta,
	\end{align*}
	which implies $y(t,\beta)\geq  y^-(t).$

Likewise, defining $y^+(t):=y+C_{E_0}(t-\tau),$  it follows that
	\begin{align*}
		y^+(t)+{\mu}_{(t)}\{(-\infty,y^+(t)]\}
		&\geq y+C_{E_0}(t-\tau)+{\mu}_{\tau}\{(-\infty,y)\}+C_{E_0}(t-\tau)\notag\\
		&\geq y+{\mu}_{\tau}\{(-\infty,y)\}\geq \beta.
	\end{align*}
	Hence, we deduce that $
	y(t,\beta)\leq y^+(t):=y+C_{E_0}(t-\tau).$	
\end{proof}
\begin{lemm}\label{uniqueness}
	Let $k(t,x)\in H^1(\mathbb{R})$ be a  conservative  solution  of the Cauchy problem \eqref{k}. Then, for any $\bar{y}(\xi)\in\mathbb{R},$  there exists
	a unique Lipschitz continuous map $t\mapsto y(t,\beta) :=y(t,\beta(t,\xi))$ which satisfies  both \eqref{yy} and \eqref{yy1}, where $y(t,\beta)$ is the solution of \eqref{yy}. Moreover,
	for any $0 \leq \tau \leq t$ one has
	\begin{align}\label{k1-k2}
		k(t,y(t))-k(\tau,y(\tau))=-\int_{\tau}^tP_x(s,y(s))ds.
	\end{align}
\end{lemm}
	\begin{proof}
	\textbf{Step 1.} Assume that $y(t)$ is the characteristic  beginning at $\bar{y}(\xi)=\xi\in\mathbb{R}$, which is defined as  $t\rightarrow y(t)=y(t,\beta(t)),$ the map $\beta(\cdot)$ is  to be determined. Let $y(t)$ be the solution of \eqref{yy} and \eqref{yy1}.
	For $t\notin\mathcal{N},$ we have
	\begin{align}\label{F0}
		\beta(t,\xi)=&y(t)+\int_0^{y(t)}k_x^2(z)dz\notag\\
		=&\int_0^{t}\frac{d}{ds}(y(s)+\int_0^{y(s)}k_x^2(s,z)dz)ds+\bar{y}(\xi)+\int_{-\infty}^{\bar{y}(\xi)}\bar{k}_x^2(z)dz\notag\\
		=&
		\bar{y}+\int_0^{\bar{y}(\xi)}\bar{k}_x^2(z)dz+\int_0^t \Gamma+\{\int_{-\infty}^{y(t)}\Big(e^{{-\lambda }t}k_x+2k_x(e^{{-\lambda }t}k^2-P-H(k))\Big)(s,z)dz\}ds.
	\end{align}
Set
	\begin{align}\label{F}
		F(t,\beta(t,\xi))\triangleq\Gamma+\int_{-\infty}^{y(t,\beta(t,\xi))}\Big(e^{{-\lambda }t}k_x+2k_x(e^{{-\lambda }t}k^2-P-H(k))\Big)dy=\xi,
	\end{align}
	and
	\begin{align}\label{F1}
		\bar{\beta}(t,\xi)\triangleq\bar{y}(\xi){+}\int_{\infty}^{\bar{y}}\bar{k}_x^2dy=\xi.
	\end{align}
According to \eqref{F0}-\eqref{F1},
it follows that
	\begin{align}\label{B}
		\beta(t,\xi)=\xi+\int_0^tF(s,\beta(s,\xi))ds.
	\end{align}
	
	\textbf{Step 2.} Since  $\|k\|_{H^1}
	=\|\bar{k}\|_{H^1},$  $y_{\beta}=\frac{1}{1+k^2_x}$ and the map  $\beta \mapsto k(t,y(t,\beta))$ is Lipschitz continuous, we get
	\begin{align*}
		F_{\beta}=\{e^{{-\lambda }t}k_x+2k_x[e^{{-\lambda }t}k^2-P-H(k)]\}y_{\beta}=\frac{e^{{-\lambda }t}k_x+2k_x[e^{{-\lambda }t}k^2-P-H(k)]}{1+k_x^2}.	
	\end{align*}
It follows that
	$$\|F_{\beta}\|_{L^{\infty}}\leq C_{E_0}.$$
	Moreover
	\begin{align}\label{eqa11}
		|F(s,{\beta}_2)-F(s,{\beta}_1)|\leq C_{E_0}|{\beta}_2-{\beta}_1|.
	\end{align}
Then, the map $\beta\mapsto F(t,\beta(t))$ is Lipschitz continuous.
	Moreover, the map $\xi \mapsto \beta(t,\xi)$ is strictly monotonic and Lipschitz continuous. From \eqref{eqa11}, we get
	\begin{align*}
		|\beta(t,{\xi}_2)-\beta(t,{\xi}_1)|&\leq |{\xi}_2-{\xi}_1|+\int_0^t|F(s,\beta(s,{\xi}_2))-F(s,\beta(s,{\xi}_2))|ds \notag\\
		&\leq |{\xi}_2-{\xi}_1|+C_{E_0}\int_0^t |\beta(t,{\xi}_2)-\beta(t,{\xi}_1)|.
	\end{align*}
The Gronwall inequality ensures that
	$$|\beta(t,{\xi}_2)-\beta(t,{\xi}_1)|\leq |{\xi}_2-{\xi}_1| e^{C_{E_0}t}.$$
Consequently,	for any  $\xi_2>\xi_1,$ we get
	\begin{align*}
		\beta(t,{\xi}_2)-\beta(t,{\xi}_1)={\xi}_2-{\xi}_1+\int_0^tF(s,\beta(s,{\xi}_2))-F(s,\beta(s,{\xi}_2))ds\geq ({\xi}_2-{\xi}_1)(1-C_{E_0}t),
	\end{align*}
	which means that  monotonicity makes sense as $t$ sufficiently small and the solution $\beta(\cdot)$ of the integral equation \eqref{B}
	depends Lipschitz continuously on the initial data.
Without loss of generality, assume that  $t$ is enough small, otherwise we can use the continuous method. In addition, the map $\xi \mapsto F(t,y(t,\beta(t,\xi)))$ is also Lipschitz continuous.
	
	\textbf{Step 3.} Owing to  $F$ is uniformly Lipschitz continuous, we can check that  existence and uniqueness of  solution of  \eqref{B} by the fixed point theorem.
	We introduce the Banach space of all
	continuous function $\beta :{\mathbb{R}}^+\rightarrow \mathbb{R}$  with weighted norm
	$$\|\beta\| :=\sup\limits_{t\geq0} e^{-2Ct}|\beta(t)|.$$
For this space, we see that the Picard map
	$$(\mathbb{P} \beta)(t)=\bar{\beta}+\int_0^tF(s,\beta(s))ds$$
	is a strict contraction. If $\|{\beta}_2-{\beta}_1\|=h>0$, we obtain
	$$|{\beta}_2(s)-{\beta}_1(s)|\leq he^{2Cs}.$$
	Then, we deduce that
	$$|(\mathbb{P} {\beta}_2)(t)-(\mathbb{P} {\beta}_1)(t)|\leq\int_0|F(s,{\beta}_2)-F(s,{\beta}_1)|ds\leq C\int_ 0^t|{\beta}_2-{\beta}_1|ds\leq |\int_0^tChe^{2Cs}ds|\leq \frac{h}{2}e^{2Ct}.$$
	This means $\|(\mathbb{P} {\beta}_2)(t)-(\mathbb{P} {\beta}_1)(t)\|\leq \frac{h}{2}.$
The contraction mapping principle guarantees that  \eqref{B}
	has a unique solution. Thanks to the arbitrary of the $T,$ we infer that the integral equation \eqref{B} has a unique solution on ${\mathbb{R}}^+.$
	
	\textbf{Step 4.} Combining  \eqref{yy} and the integral equation \eqref{B}, we infer that the uniqueness of $x(t,\beta)$ depends on the uniqueness of $\beta(t,\xi).$ From the previous analysis, the map $t\mapsto y(t)=y( t,\beta(t,\xi))$ provides the unique solution to \eqref{B}. Owing to $\beta(t)$ and $x(t, \beta(t,\xi))$ are Lipschitz continuous, then $\beta(t,\xi)$ and
	$x(t)$ are differentiable a.e.  Next, we need to prove that \eqref{F0} satisfies \eqref{yy}. Indeed, at any $\tau>0,$ we have
	
	(1) $y(\tau)$ is differentiable at $t=\tau;$
	
	(2) the measure ${\mu}_{(\tau)}$ is absolutely continuous.
	
We argue by contradiction, assume that $(1)$ does not hold, we get ${y}'(\tau)\neq e^{{-\lambda }\tau}k(\tau,y(\tau))+\Gamma$. Then,  there exists some $\epsilon_0>0,$ such that
	\begin{align}\label{5133}	
		{y}'(\tau)= e^{{-\lambda }\tau}k(\tau,y(\tau))+\Gamma+2{\epsilon}_0.
	\end{align}
Therefore, for $\delta>0$ sufficiently small, it follows that
	\begin{align}\label{5134}
		y^+(t):=y(\tau)+(t-\tau)[e^{{-\lambda }t}k(\tau,y(\tau))+\Gamma+\epsilon_0]<y(t),~for~t\in (\tau,\tau+\delta].
	\end{align}
The approximation argument guarantees that
 \eqref{0.6}  remain hold for any test function $\phi\in H^1(\mathbb{R})$ with compact support.	
For any $\epsilon>0$  enough small, we give the following functions
	\begin{equation}\label{eqa17}
		\varrho^{\varepsilon}(s,y)=\left\{
		\begin{array}{rcl}
			0, &y\leq -{\varepsilon}^{-1},\\
			y+{\varepsilon}^{-1}, & ~~-{\varepsilon}^{-1}\leq y \leq 1-{\varepsilon}^{-1},\\
			1-{\varepsilon}^{-1}(y-y(s)),  &~~~~~~ y^+(s)\leq y\leq y^+(s)+\varepsilon,\\
			0, & ~~~~~~y\geq y^+(s)+\varepsilon,
		\end{array} \right.
	\end{equation}
	\begin{equation}\label{eqa18}
		\chi^{\varepsilon}(s)=\left\{
		\begin{array}{rcl}
			0, &~~~~~~~~~~s\leq \tau-{\varepsilon}^{-1},\\
			{\varepsilon}^{-1}(s-\tau+\varepsilon), & ~~\tau-{\varepsilon}\leq s \leq \tau,\\
			1-{\varepsilon}^{-1}(s-t),  &~~ t\leq s\leq t+\varepsilon,\\
			0, & ~~~~~~~s\geq t+\varepsilon.
		\end{array} \right.
	\end{equation}
	Define
	\begin{align*}
		\phi^{\varepsilon}(s,y)=\min\{ {\varrho^{\varepsilon}(s,y),	\chi^{\varepsilon}(s)}\}.
	\end{align*}
	Let ${\phi}^{\varepsilon}$ be the test function in \eqref{0.6}. Therefore, one has
	\begin{align*}
		\int_{\mathbb{R}^+}\int_{\mathbb{R}}k^2_x\phi^{\varepsilon}_t+(e^{{-\lambda }t}k+\Gamma) k^2_x{\phi}_x^{\varepsilon}+2(e^{{-\lambda }t}k^2-P-H(k)))\phi^{\varepsilon}dxdt=0,
	\end{align*}
which means
	\begin{align}\label{3.91}
		\int_{\tau-\varepsilon}^{t+\varepsilon}\int_{-{\varepsilon}^{-1}}^{y^+(s)-\varepsilon}k^2_x\phi^{\varepsilon}_t+((e^{{-\lambda }t}k^2-P-H(k))))\phi^{\varepsilon}dydt=0.
	\end{align}
	If $t$ is sufficiently close to $\tau,$ we have
	\begin{align*}
		\lim\limits_{\varepsilon\rightarrow 0}\int_{\tau}^{t}\int_{y^+(s)-\varepsilon}^{y^+(s)+\varepsilon}k^2_x[\phi^{\varepsilon}_t+(e^{{-\lambda }t}k+\Gamma)\phi^{\varepsilon}_x](s,x)d y)ds\geq 0.
	\end{align*}
Using the fact that $e^{{-\lambda }s}k(s, y(s)) <e^{{-\lambda }\tau}k(\tau, y(\tau))+{\epsilon}_0$   and ${\phi}_x^{\epsilon}\leq 0.$ For any $s\in[\tau+\varepsilon,t-\varepsilon],$ we infer that
	\begin{align}\label{5.139}
		0={\phi}_t^{\varepsilon}+[e^{{-\lambda }\tau}k^2(\tau,y(\tau))+\Gamma]{\phi}_x^{\varepsilon}\leq {\phi}_t^{\varepsilon}+(e^{{-\lambda }s}k(s, y(s))+\Gamma){\phi}_x^{\varepsilon}.
	\end{align}
Noticing that the family of measures
	${\mu}_t$ depends continuously on $t$ in the topology of weak convergence, taking the limit of \eqref{3.91} as $\epsilon\rightarrow 0$, we obtain
	\begin{align}\label{37}
		0=&\int_{-\infty}^{y(\tau)}k^2_x(\tau,y)dy-\int_{-\infty}^{y^+(t)}k^2_x(t,y)dy+\int_{\tau}^{t}\int_{-\infty}^{y^+(s)}2k_x((e^{{-\lambda }t}k^2-P-H(k))))\phi^{\varepsilon}(s,y)dyds\notag\\
		&+\lim\limits_{\varepsilon\rightarrow 0}	\int_{\tau}^t\int_{{y^+}(s)-\varepsilon}^{{y^+}(s)+\varepsilon}k^2_x[\phi^{\varepsilon}_t+(e^{{-\lambda }t}k+\Gamma)\phi^{\varepsilon}_x](s,y)dyds\notag\\
		\geq&\int_{-\infty}^{y(\tau)}k^2_x(\tau,y)dy-\int_{-\infty}^{y^+(t)}k^2_x(t,y)dy+\int_{\tau}^{t}\int_{-\infty}^{y(s)}2k_x(((e^{{-\lambda }t}k^2-P-H(k)))))\phi^{\varepsilon}(s,y)dyds\notag\\
		&+\underbrace{\int_{\tau}^{t}\int_{y(s)}^{{y^+}(s)}2k_x(e^{{-\lambda }t}k^2-P-H(k))\phi^{\varepsilon}(s,y)dyds}_{o_1(t-\tau)}.
	\end{align}
That is
	\begin{align*}
		{\mu}_{t}\{(-\infty,{y}^+(t)]\}\geq&{\mu}_{\tau}\{(-\infty,{y}(\tau)]\}+\int_{\tau}^{t}\int_{-\infty}^{y(s)}2k_x(e^{{-\lambda }t}k^2-P-H(k))\phi^{\varepsilon}(s,y)dyds+o_1(t-\tau).
	\end{align*}
	From \eqref{5134} and  the map $t \mapsto y(t)$ is Lipschitz continuity,  we have
	\begin{align}\label{5.142}
		|o_1(t-\tau)|
		\leq& \|2(e^{{-\lambda }t}k^2-P-H(k))\|_{L^{\infty}}\int_{\tau}^t\int_{{y}(s)}^{{y}^+(s)}|k_x(s,y)|dyds\notag\\
		\leq&\|2(e^{{-\lambda }t}k^2-P-H(k)\|_{L^{\infty}}	\|k_x\|_{L^{2}}\int_{\tau}^t(y(s)-{y}^+(s))^{\frac{1}{2}}ds\notag\\
		\leq&C(t-\tau)^{\frac{3}{2}},
	\end{align}
	where $o_1(t-\tau)$ satisfies $\frac{o_1(t-\tau)}{t-\tau}\rightarrow 0$ as $t\rightarrow\tau$ and $C$ depend on  $E_0.$\\
Combining  \eqref{B}, \eqref{5134} and  \eqref{37}-\eqref{5.142},  for $t$ being close enough to $\tau$, we have
	\begin{align}\label{5.144}
		\beta(t)=&\beta(\tau)+(t-\tau)\Big(e^{{-\lambda }t}k(\tau,y(\tau))+\Gamma+\int_{-\infty}^{{y}(\tau)}2k_x[-e^{{-\lambda }t}k^2-P-H(k)](s,y)dyds\Big)
		+o_2(t-\tau)\notag\\=&y(t)+{\mu}_{t}\{(-\infty,{y}(t)]\}\notag\\
		>&y(\tau)+(t-\tau)[e^{{-\lambda }t}k(\tau,y(\tau))+\Gamma+{\epsilon}_0]+{\mu}_{\tau}\{(-\infty,{y}(\tau)]\}\notag\\
		&+\int_{\tau}^t\int_{-\infty}^{{y}(s)}2k_x[(e^{{-\lambda }t}k^2-P-H(k))]dyds+o_1(t-\tau).
	\end{align}
	with $o_2(t-\tau)$ satisfies $\frac{o_1(t-\tau)}{t-\tau}\rightarrow 0$ as $t\rightarrow\tau.$
From \eqref{5.144}, we have
	\begin{align*}
		~~~(t-&\tau)\Big(\int_{-\infty}^{{y}(\tau)}2k_x[e^{{-\lambda }t}k^2-P-H(k)](s,y)dy\Big)+o_2(t-\tau)\notag\\
		&\geq (t-\tau){\epsilon}_0+\int_{\tau}^{t}\int_{-\infty}^{{y}(s)}2k_x[e^{{-\lambda }t}k^2-P-H(k)](s,y)dyds+o_1(t-\tau).
	\end{align*}
	Dividing both sides by $t-\tau$ and letting $t\rightarrow\tau$,  one has ${\epsilon}_0<0,$ which contradicts with ${\epsilon}_0>0$. In addition, for the case ${\epsilon}_0<0,$ we follow the similar strategy as ${\epsilon}_0>0$. Hence, we conclude that $y(t)$ is differentiable at $t=\tau.$
	
	\textbf{Step 5.} It follows from  Definition \ref{def1} that
	\begin{align}\label{5.147}
		\int_{\mathbb{R}^+}\int_{\mathbb{R}}k\phi_t+\frac{(e^{{-\lambda }t}k+\Gamma)}{2} k_x{\phi}_x+P_x\phi_xdxdt+\int_{{\mathbb{R}}} \bar{k}(x){\phi}(0,x)dx=0,
	\end{align}
for any test function $\phi\in C_c^{\infty}.$ The approximation argument guarantees that the equation \eqref{5.147} remains hold, for any test function $\psi$ which is Lipschitz continuous with compact support. Note that the map $y\rightarrow k(t, y)$  is absolutely continuous and integrate
	by parts with respect to $x$. Therefore, for any $\varphi\in C_c^{\infty} $, taking $\phi={\varphi}_x$. we have
	\begin{align}\label{5.148}	
		\int_{\mathbb{R}^+}\int_{\mathbb{R}}k_x\varphi_t+(e^{{-\lambda }t}k+\Gamma) k_x{\varphi}_x-P_x\varphi_xdxdt+\int_{{\mathbb{R}}} \bar{k}_x(x){\varphi}(0,x)dx=0.
	\end{align}
	For any  $\epsilon\geq0$  sufficiently small, we give the following function
	\begin{equation*}
		\varrho^{\varepsilon}(s,y)=\left\{
		\begin{array}{rcl}
			0, &~~~~y\leq -{\varepsilon}^{-1},\\
			y+{\varepsilon}^{-1}, &~~~~~~ -{\varepsilon}^{-1}\leq y \leq 1-{\varepsilon}^{-1},\\
			1, &1-{\varepsilon}^{-1}\leq y\leq y(s),\\
			1-{\varepsilon}^{-1}(y-y(s)),  & ~~~~~~~~~y(s)\leq y\leq y(s)+\varepsilon,\\
			0 & ~~~~~~~~y\geq y(s)+\varepsilon,
		\end{array} \right.
	\end{equation*}
	and
	\begin{align*}
		\psi^{\varepsilon}(s,y)=\min\{ {\varrho^{\varepsilon}(s,y),	\chi^{\varepsilon}(s)}\},
	\end{align*}
	where $\chi^{\varepsilon}$ is defined  in \eqref{eqa18}. Let $\varphi=\psi^{\varepsilon}$ and  $\epsilon\rightarrow0$. Then, it follows from  the continuity property of function $P_x$ that
	\begin{align*}
		\int_{-\infty}^{y(t)}k_x(t,y)dy&=\int_{-\infty}^{y(\tau)}k_x(\tau,y)dy-\int_{\tau}^tP_x(s,y(s))ds\notag\\
&~~~+\lim\limits_{\varepsilon\rightarrow 0}	\int_{\tau-\varepsilon}^{t+\varepsilon}\int_{{y}(s)}^{{y}(s)+\varepsilon}k_x[{\psi}_t^{\varepsilon}+(e^{{-\lambda }t}k+\Gamma){\psi}_x^{\varepsilon}]dyds.
	\end{align*}
Hence, it is shown that
	\begin{align}\label{537}
		&\lim\limits_{\varepsilon\rightarrow 0}	\int_{\tau-\varepsilon}^{t+\varepsilon}\int_{{y}(s)}^{{y}(s)+\varepsilon}k_x[{\psi}_t^{\varepsilon}+(e^{{-\lambda }t}k+\Gamma){\psi}_x^{\varepsilon}]dyds\notag\\
&=\lim\limits_{\varepsilon\rightarrow 0}	\Big(\int_{\tau-\varepsilon}^{\tau}+\int_{\tau}^{t}+\int_{t}^{t+\varepsilon}\Big)\int_{{y}(s)}^{{y}(s)+\varepsilon}k_x[{\psi}_t^{\varepsilon}+(e^{{-\lambda }t}k+\Gamma){\psi}_x^{\varepsilon}]dyds=0.
	\end{align}
	
	First, we claim that
	\begin{align*}
		\lim\limits_{\varepsilon\rightarrow 0}	\int_{\tau}^t\int_{{y}(s)}^{{y}(s)+\varepsilon}k_x[{\psi}_t^{\varepsilon}+(e^{{-\lambda }t}k+\Gamma){\psi}_x^{\varepsilon}]dyds=0.
	\end{align*}
	Taking advantage of  Cauchy's inequality and $k_x\in L^2$, one has
	\begin{align*}
		&|\int_{\tau-\varepsilon}^{t}\int_{{y}(s)}^{{y}(s)+\varepsilon}k_x[{\psi}_t^{\varepsilon}+(e^{{-\lambda }t}k+\Gamma){\psi}_x^{\varepsilon}]dyds|\notag\\&\leq \int_{\tau-\varepsilon}^{t}\Big(\int_{{y}(s)}^{{y}(s)+\varepsilon}|k_x|^2dy\Big)^{\frac{1}{2}}	\Big(\int_{{y}(s)}^{{y}(s)+\varepsilon}[{\psi}_t^{\varepsilon}+(e^{{-\lambda }t}k+\Gamma){\psi}_x^{\varepsilon}]^2dy\Big)^{\frac{1}{2}}ds.	
	\end{align*}
	Define
	\begin{align*}
		{\pi}_{\varepsilon}(s)=\Big(\sup\limits_{x\in \mathbb{R}} \int_{{y}(s)}^{{y}(s)+\varepsilon}k_x^2(s,y)dy\Big)^{\frac{1}{2}}.
	\end{align*}
	 Note that the function  ${\pi}_{\epsilon}(s)$ is uniformly bounded for  $\epsilon$ and $\lim\limits_{\varepsilon\rightarrow 0}{\pi}_{\varepsilon}(s)=0$
	  almost every time $t.$   Hence, it follows from the dominated convergence theorem that
	\begin{align}\label{5.157}
		\lim\limits_{\varepsilon\rightarrow 0}\int_{\tau}^{t}\Big(\sup\limits_{x\in \mathbb{R}} \int_{{y}(s)}^{{y}(s)+\varepsilon}k_x^2(s,y)dy\Big)^{\frac{1}{2}}ds\leq 	\lim\limits_{\varepsilon\rightarrow 0}\int_{\tau}^{t}{\pi}_{\varepsilon}(s)ds=0.
	\end{align}
	In addition,  for all  $s\in [\tau, t]$, one can get from the definition of $\psi_{\epsilon}$ that
\begin{align}\label{5.158}
		{\psi}_x^{\varepsilon}(s,y)=-{\varepsilon}^{-1}, \quad {\psi}_t^{\varepsilon}(s,y)+(e^{{-\lambda }s}k(s,y(s))+\Gamma){\psi}_x^{\varepsilon}(s,y)=0,
	\end{align}
	with $y(s) < y < y(s) +\epsilon.$ Then, it follows from \eqref{5.158} that
	\begin{align}\label{5159}
		&\int_{{y}(s)}^{{y}(s)+\varepsilon}({\psi}_t^{\varepsilon}(s,y)+(e^{{-\lambda }s}k(s,x(s))+\Gamma){\psi}_x^{\varepsilon}(s,y))^2dy\notag\\&={\varepsilon}^{-2}\int_{{y}(s)}^{{y}(s)+\varepsilon}(k(s,y)-{k(s,y(s))})^2dy\notag\\
		&\leq{\varepsilon}^{-1}\Big(\max\limits_{y(s) < y < y(s) +\epsilon}|k(s,y)-{k(s,y(s))}|\Big)^2\notag\\&\leq {\varepsilon}^{-1}\Big(\int_{{y}(s)}^{{y}(s)+\varepsilon}|k(s,y)-{k(s,y(s))}|dy\Big)^2\notag\\
		&\leq{\varepsilon}^{-{\frac{1}{2}}}({\varepsilon}^{{\frac{1}{2}}}\|k_x(s)\|_{H^{1}})^2\leq C\|k(s)\|_{H^{1}}.
	\end{align}
	Using \eqref{5.157} and \eqref{5159},  we obtain
	\begin{align}\label{538}
		\lim\limits_{\varepsilon\rightarrow 0}	\int_{\tau}^{t}\int_{{y}(s)}^{{y}(s)+\varepsilon}k_x[{\psi}_t^{\varepsilon}+(e^{{-\lambda }s}k+\Gamma){\psi}_x^{\varepsilon}](s,y)dyds=0.
	\end{align}
The Cauchy-Schwartz inequality and \eqref{5.158} entail that
	\begin{align}\label{539}
		&\lim\limits_{\varepsilon\rightarrow 0}	(\int_{\tau-\varepsilon}^{t}+\int_{t}^{t+\varepsilon})\int_{{y}(s)}^{{y}(s)+\varepsilon}k_x[{\psi}_t^{\varepsilon}+(e^{{-\lambda }s}k+\Gamma){\psi}_x^{\varepsilon}](s,y)dyds\notag\\
		&\leq\lim\limits_{\varepsilon\rightarrow 0}	(\int_{\tau-\varepsilon}^{t}+\int_{t}^{t+\varepsilon})\Big(\int_{{y}(s)}^{{y}(s)+\varepsilon}|k_x|^2dy\Big)^{\frac{1}{2}}(\int_{{y}(s)}^{{y}(s)+\varepsilon}[{\psi}_t^{\varepsilon}+(e^{{-\lambda }s}k+\Gamma){\psi}_x^{\varepsilon}]^2dy\Big)^{\frac{1}{2}}ds\notag\\
		&\leq\lim\limits_{\varepsilon\rightarrow 0}2\epsilon\|k(s)\|_{H^{1}}\Big(\int_{{y}(s)}^{{y}(s)+\varepsilon}2{\epsilon}^{-2}\|k\|_{L^{\infty}}^2dy\Big)^{\frac{1}{2}}\notag\\
		&\leq \lim\limits_{\varepsilon\rightarrow 0}C{\epsilon}^{\frac{1}{2}}=0.
	\end{align}
	Combining \eqref{538}-\eqref{539}, we arrive at \eqref{537}.
	
	\textbf{Step 6.} Using the uniqueness of $\beta(t,\xi),$ we can deduce that the uniqueness of $y(t,\xi).$
\end{proof}
The following lemma is to prove the Lipschitz continuity of $k$ with respect to $t$ under the Lagrange coordinates.
		\begin{lemm}\label{l5.3}
			Let $k = k(t, x)$ be a conservative solution to  \eqref{k}.  Then the map
			$(t,\beta)\mapsto k(t,y(t,\beta))$ is Lipschitz continuous  with a constant depending only on the
			norm $\|\bar{k}\|_{H^1}.$
		\end{lemm}
		\begin{proof}
			Combining  \eqref{F} and \eqref{B}, we get
			\begin{align*}
				|k(t,y(t,\bar{\beta}))-k(\tau,\bar{\beta})|\leq& |k(t,y(t,\bar{\beta}))-k(t,y(t,{\beta}(t)))|+|k(t,y(t,{\beta}(t)))-k(\tau,y(\tau,{\beta}(\tau)))|\notag\\
				\leq&\frac{1}{2}|{\beta}(t)-\bar{\beta}|+(t-\tau)\|P_x\|_{L^{\infty}}\notag\\
				\leq& (t-\tau)(\frac{1}{2}\|F\|_{L^{\infty}}+\|P_x\|_{L^{\infty}}),
			\end{align*}
			which implies the map
			$(t,\beta)\mapsto k(t,x(t,\beta))$ is Lipschitz continuous.
		\end{proof}
		\begin{lemm}
			Let $k\in H^1(\mathbb{R})$ and define the convolution $P$ being as in \eqref{PP}. Then $P_x$ is absolutely continuous and satisfies
			\begin{align}\label{p33}
				P_{xx}=P-\Big(-H(k)+e^{-\lambda t}k^2+\frac  {e^{{-\lambda }t}} {2}k_x^2\Big).
			\end{align}
		\end{lemm}
		\begin{proof}
			The function $\psi(s)=\frac{e^{-|x|}}{2}$ satisfies the distributional identity
			$$D^2_x\psi=\psi-\delta_0.$$
			Thanks to $\delta_0$ denotes a unit Dirac mass at the origin. Thus, for all function $f\in L^1(\mathbb{R}),$ the convolution satisfies
			$$D^2_x(\psi \ast f)=\psi \ast f-f.$$
			Choosing $f=-H(k)+e^{-\lambda t}k^2+\frac  {e^{{-\lambda }t}} {2}k_x^2,$ we obtain the desired result.
		\end{proof}
\subsection{Proof of uniqueness}
We need to seek  a good characteristic, and employ haw the gradient $k_x$ of a conservative solution varies along the good characteristic, and complete the  proof of uniqueness.
		\begin{proof}\textbf{Step 1.} Lemmas \ref{uniqueness}-\ref{l5.3} ensure that the map $(t,\xi)\mapsto (y,k)(t,\xi)$ and $\xi\mapsto F(t,\xi)$ are Lipschitz continuous. Thanks to  Rademacher's theorem, the partial derivatives $y_t, y_{\xi}, k_t,k_{\xi}$ and $F_{\xi}$ exist almost everywhere. Moreover,  $y(t,\xi)$
			is the unique solution to (5.1), and the following holds.
			
			$\textbf{(GC)}$ For  a.e. $\xi$ and a.e. $t\geq0$, the point $(t,\beta(t,\bar{\beta}))$ is a Lebesgue point for the partial derivatives $y_t,y_{\xi},k_t,k_{\xi}$
			and $F_{\xi}$. Moreover, $y_{\xi}(t,\xi)>0$ for $a.e. ~t\geq0.$
			
			If 	$\textbf{(GC)}$ holds, then  $t \rightarrow y(t,\xi)$ is a good characteristic.
			
			\textbf{Step 2.}  We now construct an ODE to describe that the quantities $k_{\xi}$ and $x_{\xi}$ vary along a good characteristic. Supposing that $t,~\tau\notin\mathcal{N},$ and $y(t,\xi)$ is a good characteristic, we then have
			\begin{align*}
				y(t,\beta(t,\xi))=\bar{y}(\xi)+\int_0^t(e^{-\lambda s}k(s,\beta(s,\xi))+\Gamma) ds.
			\end{align*}
			Differentiating the above equation with respect to $\xi$, we deduce that
			\begin{align}\label{p35}
				y_{\xi}=\bar{y}_{\xi}(\xi)+\int_0^tk_{\xi}(s,\xi)d\xi.
			\end{align}
			Likewise, we have
			\begin{align}\label{p36}
				k_{\xi}=\bar{k}_x(\bar{y}(\xi))\bar{y}_{\xi}-\int_0^t G_{\xi}(s,\xi)d\xi.
			\end{align}
From \eqref{p35}-\eqref{p36}, we end up with
			\begin{equation}\label{qp37}
					\left\{\begin{aligned}
			&y_{t\xi}=e^{-\lambda t}k_{\xi},\\
			&k_{t\xi}=-G_{\xi}.
		\end{aligned}\right.
			\end{equation}

			\textbf{Step 3.} We now  return to the original coordinates $(t, x)$ and derive an evolution equation for the
			partial derivative $k_x$ along a “good” characteristic curve.
			For a fixed point $(t, x)$ with $t\notin \mathcal{N}$. Suppose
			that $\bar{x}$ is a Lebesgue point for the map $x\rightarrow k_x(t, x)$, and $\xi$ satisfies $x = y(t, \xi),$ and suppose
			that $t\rightarrow y(t,\xi)$ is a good characteristic, which implies $\textbf{(GC)}$ holds. From \eqref{yy} we have
			\begin{align}\label{p38}
				y_{\beta}(t,\beta)=\frac{1}{1+k^2_x(t,y)}>0,\ \beta(t,\xi)>0,
			\end{align}
			which implies that $y_{\xi}(t,\xi)>0.$
		
		Hence,  the partial derivative
			$k_x$ can be calculated as shown below
			\begin{align*}
				k_x(t,y(t,\beta(t;\tau,\bar{\beta})))=\frac{k_{\xi}(t,y(t,\beta(t,{\xi})))}{y_{\xi}(t,\beta(t;\bar{\xi}))}.
			\end{align*}
		Applying \eqref{qp37} to describe the evolution of $k_{\xi}$ and $y_{\xi},$ we  infer that the map $t\rightarrow k_x(t,y(t,\beta(t,{\xi})))$ is absolutely continuous. It follows that
			\begin{align*}
				\frac{d}{dt}k_x(t,y(t,\beta(t;\tau,{\xi})))&=\frac{d(\frac{k_{\xi}}{y_{\xi}})}{dt}=\frac{y_{\xi}F_{\xi}-e^{-\lambda t}{k_{\xi}}^2}{{y_{\xi}}^2}.
			\end{align*}
			Hence, we conclude that as long as $y_{\beta}\neq 0$, the map $t\rightarrow k_x$  is absolutely continuous.
			
			\textbf{Step 4.} Let
			\begin{align*}
				&K(t,\xi)=k(t,y(t,\xi)),\quad V(t,\xi)=\frac{k_x^2\circ y}{1+k_x^2\circ y},\notag\\
				&W(t,\xi)=\frac{k_x\circ y}{1+k_x^2\circ y},\quad Q(t,\xi)=({1+k_x^2\circ y})\cdot  y_{\xi}.
			\end{align*}
	From which it follows that	
		\begin{equation}
		\left\{\begin{aligned}
			y_t&=e^{{-\lambda }t}K+\Gamma,\\
			K_t&=-G,\\
			V_t&=2W\Big(e^{{-\lambda }t}K^2(1-V)-H(K)(1-V)-\frac {e^{{-\lambda }t}} {2}V-{P}(1-V)\Big),\\
			W_t&=(1-2V)\Big(e^{{-\lambda }t}K^2(1-V)-H(K)(1-V)-\frac {e^{{-\lambda }t}} {2}V-{P}(1-V)\Big),\\
			Q_t&=2WQ\Big(\frac{e^{{-\lambda }t}}{2}+e^{{-\lambda }t}K^2-H(K)-{P}\Big).\\
		\end{aligned} \right.\label{Ky1}
	\end{equation}	
			For any $\xi\in \mathbb{R},$ we deduce that the following initial  conditions
		\begin{equation}
			\left\{\begin{aligned}
				\int_{0}^{\bar{y}}\bar{k}_x^2&dx+\bar{y}(\xi)=\xi,\\
				\bar{K}(\xi)&=\bar{k}\circ \bar{y}(\xi),\\
				\bar{V}(\xi)&=\frac{\bar{k}_x^2\circ \bar{y}}{1+\bar{k}_x^2\circ \bar{y}(\xi)},\\
				\bar{W}(\xi)&=\frac{\bar{k}_x\circ \bar{y}(\xi)}{1+\bar{k}_x^2\circ \bar{y}(\xi)},\\
				\bar{Q}(\xi)&=(1+\bar{k}_x^2\circ \bar{y})\bar{y}_{\xi}(\xi)=1.\\
			\end{aligned}\right.\label{K01}
		\end{equation}
			Making use of  all coefficients is Lipschitz continuous and the previous steps again, the system  \eqref{Ky1}-\eqref{K01} has a unique globally solution.
			
			\textbf{Step 5.} Let $k$ and $\tilde{k}$ be two conservative weak solution of  \eqref{k} with the same
			initial data $\bar{k}\in H^1(\mathbb{R})$
			. For $a.e. ~t\geq0,$  the corresponding Lipschitz continuous maps $\xi \mapsto y(t,\xi), {\xi} \mapsto \tilde{y}(t,\beta)$ are
			strictly increasing. Hence they have continuous inverses, say $x\mapsto{y}^{-1}(t,x),
			x\mapsto{\tilde{y}}^{-1}(t,x)$.
			Thus, we deduce that
			$$y(t,\xi)=\tilde{y}(t,\xi),\  k(t,y(t,\xi))=\tilde{k}(t,{y}(t,\xi)).$$
			Moreover, for $a.e.~ t\geq0$, we have
			$$k(t,x)=k(t,y(t,\xi))=\tilde{k}(t,\tilde{y}(t,\xi))=\tilde{k}(t,x).$$
			Then, we finish the proof of Theorem \ref{uniqueness}.
		\end{proof}	
\paragraph{Proof of Theorem \ref{th2.3}.}  Theorem \ref{global k} and Theorem \ref{th2.4}  ensure that the equation \eqref{k} has a unique globally conservative solution. $~~~~~~~~~~~~~~~~~~~~~~~~~~~~~~~~~~~~~~~~~~~~~~~~~~~~~~~~~~~~~~~~~~~~~~~~~~~~~~~~~~~~~~~~~~\square$

		\noindent\textbf{Acknowledgements.}
	This work was partially supported by NNSFC (Grant No. 12171493), the FDCT (Grant No. 0091/2018/A3), the Guangdong Special Support Program (Grant No.8-2015).

\end{document}